\documentclass[11pt]{article}

\usepackage[a4paper,margin=2.3cm]{geometry}
\usepackage[utf8]{inputenc}
\usepackage[T1]{fontenc}
\usepackage[english]{babel}
\usepackage{amsmath,amssymb,amsthm,mathtools}
\usepackage{enumitem}
\usepackage{booktabs}
\usepackage{array}
\usepackage{xcolor}
\usepackage[
appendix=append,
bibliography=common
]{apxproof}
\usepackage[
colorlinks=true,
linkcolor=blue!50!black,
citecolor=green!40!black,
urlcolor=blue!50!black
]{hyperref}
\usepackage[
capitalize,
nameinlink
]{cleveref}

\usepackage{comment}

\theoremstyle{plain}

\newtheorem{theorem}{Theorem}[section]

\newtheoremrep{lemma}[theorem]{Lemma}
\newtheoremrep{proposition}[theorem]{Proposition}
\newtheoremrep{corollary}[theorem]{Corollary}

\theoremstyle{definition}

\newtheorem{definition}[theorem]{Definition}
\newtheorem{example}[theorem]{Example}

\theoremstyle{remark}

\newtheorem{remark}[theorem]{Remark}

\theoremstyle{plain}

\newtheorem{assumption}{Assumption}

\crefname{assumption}{Assumption}{Assumptions}

\newcommand{\R}{\mathbb{R}}
\newcommand{\N}{\mathbb{N}}
\newcommand{\Cbb}{\mathbb{C}}
\newcommand{\Hcal}{\mathcal{H}}
\newcommand{\Ucal}{\mathcal{U}}
\newcommand{\Ycal}{\mathcal{Y}}
\newcommand{\Acal}{\mathcal{A}}
\newcommand{\Lcal}{\mathcal{L}}
\newcommand{\Dcal}{\mathcal{D}}

\newcommand{\Ecal}{\mathcal{E}}

\newcommand{\HP}{\mathcal{H}_P}
\newcommand{\norm}[1]{\left\|#1\right\|}
\newcommand{\normH}[1]{\left\|#1\right\|_{\Hcal}}
\newcommand{\normQ}[1]{\left\|#1\right\|_{\HP}}

\newcommand{\zhat}{\hat z}

\newcommand{\Hu}{\Hcal_u}
\newcommand{\Hs}{\Hcal_s}
\newcommand{\Pu}{P_u}
\newcommand{\Ps}{P_s}
\newcommand{\Au}{\Acal_u}
\newcommand{\As}{\Acal_s}

\DeclareMathOperator{\spanop}{span}

\usepackage{authblk}

\title{Observer-Based Event-Triggered Control of Infinite-Dimensional Linear Systems}

\author[1]{Nail Baloul\thanks{Emails:
		\texttt{nail.baloul@enpc.fr},
		\texttt{baudouin@laas.fr},
		\texttt{sylvain.ervedoza@math.u-bordeaux.fr},
		\texttt{thibault.liard@unilim.fr},
		\texttt{pierre.lissy@enpc.fr}.}}

\author[2]{Lucie Baudouin}

\author[3]{Sylvain Ervedoza\thanks{S.~Ervedoza is supported by TRECOS ANR-20-CE40-0009,
		NumOpTes ANR-22-CE46-0005 and CHAT ANR-24-CE40-5470.}}

\author[4]{Thibault Liard\thanks{T.~Liard is supported by ANR-24-CE46-7295.}}

\author[1]{Pierre Lissy\thanks{N.~Baloul and P.~Lissy are supported by ANR-22-CPJ2-0138-01.
		P.~Lissy is supported by ANR-26-CE40-0732-01.}}

\affil[1]{CERMICS, CNRS, ENPC, Institut Polytechnique de Paris, Marne-la-Vallée, France}

\affil[2]{Univ Toulouse, CNRS, LAAS-CNRS, Toulouse, France}

\affil[3]{Institut de Mathématiques de Bordeaux, Université de Bordeaux, CNRS, Bordeaux, France}

\affil[4]{XLIM, Université de Limoges, CNRS, Limoges, France}

\date{\today}

\begin{document}
	
	\maketitle
	
	\begin{abstract}
		
		We investigate observer-based event-triggered stabilization 
		of infinite-dimensional linear systems, such as those governed by partial differential equations.  A Luenberger observer and a feedback gain are assumed to stabilize the closed loop in continuous time, and we study what happens when the control input is only updated at instants generated by a dynamic event-triggering mechanism. We will present several event-triggering mechanisms, almost all of them relying on the observer state and the measured output, and thus implementable when the full state is not measured. A triggering rule compares the deviation of a monitored signal of the observer since the last update with a threshold made of two terms. The first one is a relative term, proportional to the current size of the signal, and the second one is a positive internal variable, which does not depend on the observer state. In particular, the case where the threshold reduces to the internal variable alone is covered.
		
		Under suitable assumptions, we establish the well-posedness of the closed-loop system and the absence of Zeno behavior for general choices of threshold, and we prove its exponential stability under a small-gain condition on the relative term of the threshold. We further analyze the existence of a dwell-time, namely a positive lower bound on the time elapsed between two consecutive triggering instants. When the sampled dynamics is bounded, as in finite-dimensional systems or systems with a finite-dimensional unstable part, we obtain a dwell-time independent of the initial data, together with exponential stability. For unbounded generators,
		we exhibit a system satisfying all our assumptions for which no triggering rule can achieve both a dwell-time independent of the initial data and bounded trajectories.
		
		No structural assumption such as skew-adjointness is imposed on the operator generating the system dynamics, which broadens the applicability of the approach to systems with quasi-dissipative or non-conservative behavior.
		
	\end{abstract}
	
	\tableofcontents
	
	\section{Introduction}\label{sec:intro}
	
	\subsection{Motivation and context}
	
	Event-triggered control has emerged over the past two decades as a trend in feedback stabilization, particularly in control systems where communication constraints are prevalent. This paradigm is meant to reduce communication and computation load (and hence energy consumption) while preserving closed-loop stability. Instead of updating the control continuously (or periodically), the control is updated on demand, when a measured quantity crosses a threshold \cite{Tabuada2007,HeemelsJohanssonTabuada2012}. Three properties must then be guaranteed. The closed loop must be well defined (the triggering instants must be generated consistently with the trajectory), the triggering instants must not accumulate in finite time (absence of Zeno behavior), and the closed loop must be stable as expected. A fourth, stronger, property, often desirable in practice, is the existence of a dwell-time, that is, a positive lower bound on all inter-event times, ideally independent of the initial condition, which bounds by below the transmission rate.
	
	In finite dimension, these properties follow from a simple argument. Between two events, the control is frozen, and the deviation of the state from its last sampled value grows at most linearly in the elapsed time $h$, at a rate proportional to the size of the sampled state, since $\norm{e^{hA}-I}=O_{h\to 0^+}(h)$ for a square matrix $A$ (see \cref{lem:bounded-incr}). 
	
	Extending this approach to infinite-dimensional systems, such as those described by partial differential equations (PDE), is delicate because of the presence of unbounded operators, semigroup dynamics and the coexistence of stable and unstable spectral modes.
	Indeed, the generator $\Acal$ of the free dynamics is unbounded, and the previous argument no longer applies. The operator $e^{h\Acal}-I$ does not tend to $0$ in operator norm as $h\to0$. Consequently, at arbitrarily small time intervals, some states may exhibit significant changes relative to their initial norm (see \cref{lem:nonunif}). High-frequency components evolve on arbitrarily short time scales, and the inter-event times may depend on these frequencies. Moreover, in many situations, the full state is not measured and the control must be computed from an observer. The triggering rule can then only use the observer state and the measured output. This creates a second difficulty. The observer is driven not only by the sampled control, but also by the output estimation error injected through the observer gain. 
	
	Our objective is not to design a stabilizing feedback. We start from a continuous-time observer-based feedback which is known to stabilize the system exponentially, and we study the effect of an event-triggered implementation of this feedback on well-posedness, Zeno behavior, exponential stability and dwell-time, in an abstract framework covering a large class of PDE.
	
	\subsection{The control system and the standing assumptions}
	
	Let $\Hcal$, $\Ucal$ and $\Ycal$ be real or complex Hilbert spaces, respectively the state, input and output spaces.\\ 
	We start from the free system $\dot z=\Acal z$, under the following assumption.
	\begin{assumption}\label{A:gen}
		$\Acal\colon\Dcal(\Acal)\subset\Hcal\to\Hcal$ is the infinitesimal generator of a strongly continuous semigroup (for short $C_0$-semigroup) $e^{t\Acal}\coloneqq(e^{t\Acal})_{t\ge0}$ on $\Hcal$.
	\end{assumption}
	\Cref{A:gen} means that the free system is well posed in $\Hcal$. The generator $\Acal$ may be unbounded and even unstable. We will use the following classical bound on the growth of the free semigroup (see \cite[Proposition I.5.5]{EngelNagel2000}). There exist $C_\Acal\ge1$ and $\omega_\Acal\in\R$ such that
	\begin{equation}\label{eq:growth}
		\norm{e^{t\Acal}}_{\Lcal(\Hcal)}\le C_\Acal e^{\omega_\Acal t},\qquad t\ge0.
	\end{equation}
	
	The control acts through an operator $B$ and the measured output is $y=Cz$. To stabilize the system from this output, we use a Luenberger observer \cite{Luenberger1964,Luenberger1971}, and we apply a state feedback to the estimated state. In continuous time, this gives the observer-based controller
	\begin{equation}\label{eq:continuous}
		\begin{cases}
			\dot z(t)=\Acal z(t)+Bu(t), & t\ge0,\\
			\dot{\zhat}(t)=\Acal\zhat(t)+Bu(t)+L\big(C\zhat(t)-y(t)\big), & t\ge0,\\
			u(t)=K\zhat(t), & t\ge0,\\
			y(t)=Cz(t), & t\ge0.
		\end{cases}
	\end{equation}
	Here $B$ is the control operator, associated with the available actuators, $C$ is the observation operator, associated with the available measurements, $K$ is the control gain and $L$ is the injection gain, both to be chosen. We will make the following assumptions.
	\begin{assumption}\label{A:BK}
		$B\in\Lcal(\Ucal,\Hcal)$ and $K\in\Lcal(\Hcal,\Ucal)$, i.e. $B,K$ are bounded linear operators.
	\end{assumption}
	\begin{assumption}\label{A:Kstab}
		The operator $\Acal+BK$, with domain $\Dcal(\Acal)$, generates an exponentially stable $C_0$-semigroup $(e^{t({\Acal+BK})})_{t \geq 0}$ on $\Hcal$, i.e. there exist $M_K\ge1$ and $\alpha_K>0$ such that
		\[
		\norm{e^{t(\Acal+BK)}}_{\Lcal(\Hcal)}\le M_Ke^{-\alpha_Kt}\qquad \text{for every }t\ge0.
		\]
	\end{assumption}
	\begin{assumption}\label{A:CL}
		$C\in\Lcal(\Hcal,\Ycal)$ and $L\in\Lcal(\Ycal,\Hcal)$.
	\end{assumption}
	\begin{assumption}\label{A:Lstab}
		The operator $\Acal+LC$, with domain $\Dcal(\Acal)$, generates an exponentially stable $C_0$-semigroup $(e^{t({\Acal+LC})})_{t \geq 0}$ on $\Hcal$, i.e. there exist $M_L\ge1$ and $\alpha_L>0$ such that
		\[
		\norm{e^{t(\Acal+LC)}}_{\Lcal(\Hcal)}\le M_Le^{-\alpha_Lt}\qquad \text{for every }t\ge0.
		\]
	\end{assumption}
	Under \cref{A:gen,A:BK,A:CL}, the bounded perturbation theorem \cite[Theorem III.1.3]{EngelNagel2000} already guarantees that ${\Acal+BK}$ and ${\Acal+LC}$ generate $C_0$-semigroups on $\Hcal$ with domain $\Dcal(\Acal)$. \Cref{A:Kstab,A:Lstab} only concern their decay. They express that $K$ is a stabilizing feedback for the pair $(\Acal,B)$ and that $L$ is an output injection making the observer converge. No structural property of $\Acal$ itself, such as skew-adjointness, dissipativity, analyticity or compactness of the resolvent, is assumed.
	
	\paragraph{The separation principle.} Let $e\coloneqq\zhat-z$ be the estimation error. Since $C\zhat-y=Ce$, \eqref{eq:continuous} is equivalent to
	\[
	\dot{\zhat}=({\Acal+BK})\zhat+LCe,\qquad \dot e=({\Acal+LC}) e .
	\]
	Under \cref{A:gen,A:CL,A:Lstab}, the error satisfies $\normH{e(t)}\le M_Le^{-\alpha_Lt}\normH{e(0)}$. Moreover, using the variation-of-constants formula (see \eqref{eq:mild} below), the identity
	\begin{equation*}
		\int_0^te^{s \beta}\,ds=\frac1{\beta}\left(e^{\beta t} - 1 \right)\quad \text{for every } t\ge0,
	\end{equation*}
	with $\beta=\alpha_K - \alpha_L$, and \cref{A:BK,A:Kstab},
	we obtain for $\alpha_K\neq\alpha_L$ and for any $t\geq 0$,
	\[
	\begin{aligned}
		\normH{\zhat(t)}
		&=\normH{e^{t(\Acal+BK)}\zhat(0)
			+\int_0^t e^{(t-s)(\Acal+BK)}LCe(s)\,ds}\\
		&\le
		M_K e^{-\alpha_K t}\normH{\zhat(0)}
		+M_KM_L\norm{LC}\normH{e(0)}\frac{e^{-\alpha_L t}-e^{-\alpha_K t}}{\alpha_K-\alpha_L}.
	\end{aligned}
	\]
	If $\alpha_K=\alpha_L$, then $(\alpha_K-\alpha_L)^{-1}\left(e^{-\alpha_L t}-e^{-\alpha_K t}\right)$ is replaced by $t e^{-\alpha_K t}$. Hence the continuous-time loop \eqref{eq:continuous} is exponentially stable. This is an instance of the separation principle, here in the setting of bounded control and observation operators (see \cite{CurtainZwart1995} for observer-based controllers in infinite dimensions).
	The question addressed in our study is whether this exponential stability is preserved under an event-triggered implementation of the control.
	
	\paragraph{The event-triggered closed loop.} The control is now only updated at instants $t_k$ such that  $0\eqqcolon t_0<t_1<\cdots$, generated by the triggering mechanism of \cref{sec:intro-rule}, and it is held constant between two consecutive instants. This is called zero-order hold, see for instance \cite{HeemelsJohanssonTabuada2012}. The closed loop reads
	\begin{equation}\label{eq:system}
		\begin{cases}
			\dot z(t)=\Acal z(t)+Bu(t), & t\ge0,\\
			\dot{\zhat}(t)=\Acal\zhat(t)+Bu(t)+L\big(C\zhat(t)-y(t)\big), & t\ge0,\\
			u(t)=K\zhat(t_k), & t\in[t_k,t_{k+1}),\\
			y(t)=Cz(t), & t\ge0,\\
			z(0)=z_0,\quad \zhat(0)=\zhat_0, &
		\end{cases}
	\end{equation}
	where $z_0,\zhat_0\in\Hcal$ are the initial conditions of the state of the system and of its observer. The observer correction term $L(C\zhat-y)$ is updated continuously from the real-time measurement $y$, and only the control $K\zhat$ is sampled. Sampling the measurement as well, possibly at different instants, is not considered here. The estimation error would then depend on the triggering instants, and the decoupling between the continuous estimation error  $e$ and the observer state $\zhat$  described in \eqref{eq:obs-err}, which is used throughout the paper, would be lost.
	
	\paragraph{Observer--error formulation.} With $e=\zhat-z$ and $e_0\coloneqq\zhat_0-z_0$, \eqref{eq:system} is equivalent to
	\begin{equation}\label{eq:obs-err}
		\begin{cases}
			\dot{\zhat}(t)=\Acal\zhat(t)+BK\zhat(t_k)+LCe(t), & t\in[t_k,t_{k+1}),\\
			\dot e(t)=(\Acal+LC) e(t), & t\ge0,\\
			\zhat(0)=\zhat_0,\quad e(0)=e_0. &
		\end{cases}
	\end{equation}
	The control terms cancel in the error equation, so that $e(t)=e^{t({\Acal+LC})}e_0$ does not depend on the triggering instants. In particular, the estimation error evolves continuously in time. This is the reason why it is more convenient to work with the pair $(\zhat,e)$ rather than $(z,\zhat)$. The error is an explicit and exponentially decaying input of the observer equation (see \cref{lem:error}).
	
	We call the measured signal $Ce=C\zhat-y$ the output residual.
	It is the part of the measurement that is not predicted by the observer. With a slight abuse, we also call output residual its injected version $LCe=L(C\zhat-y)$, which is the correction term of the observer. The output residual term is available in real time, since it only depends on the observer and on the measurement, whereas $e(t)$ itself is unknown in general (since $C$ is in general not injective).
	
	\subsection{The event-triggering mechanism}\label{sec:intro-rule}
	
	\paragraph{Monitored signal.} 
	To handle at once full-state sampling and sampling of a finite-dimensional component of the observer $\hat z$ of the state (\cref{sec:dwell-bounded}),  we monitor a signal $P\zhat(t),$ where $P\in\Lcal(\Hcal,\HP)$, and $\HP$ is a Hilbert space. 
	In the following, $P$ is thought as a projection operator. For instance, full state sampling of $\hat z$ corresponds to $\HP=\Hcal$, $P=I$. \\
	We assume that the operator $BK$ admits a factorization under the form
	\begin{equation}\label{eq:factor}
		BK=DP, 
	\end{equation}
	where $D$ is an operator from $\mathcal{L}(\HP, \Hcal)$.
	By \eqref{eq:factor}, the sampling perturbation can be written as
	\begin{equation*}
		BK(\zhat(t_k)-\zhat(t))=DP (\zhat(t_k)-\zhat(t)), \quad t\in[t_k,t_{k+1}).
	\end{equation*}
	
	\paragraph{The triggering rule.} Given $t_k<\infty$, the next instant is given by
	\begin{equation}\label{eq:rule}
		t_{k+1}\coloneqq\sup\Big\{t>t_k,\quad \normQ{P\zhat(\tau)-P\zhat(t_k)}\le\Theta(\tau)\ \ \text{for all }\tau\in[t_k,t)\Big\}\in(t_k,\infty],
	\end{equation}
	with the triggering threshold
	\begin{equation}\label{eq:threshold}
		\Theta(\tau)\coloneqq\gamma\normQ{P\zhat(\tau)}+m(\tau),\qquad \gamma\ge0 .
	\end{equation}
	We will prove later  (see \cref{prop:wp-tech}) that under suitable assumptions the set defining \eqref{eq:rule}, i.e. $\{t>t_k,~ \normQ{P\zhat(\tau)-P\zhat(t_k)}\le\Theta(\tau)\ \text{for all }\tau\in[t_k,t)\}$, is never empty.\\
	The threshold has two terms.
	\begin{itemize}[leftmargin=*]
		\item The relative term $\gamma\normQ{P\zhat}$ is proportional to the current size of the monitored signal.
		\item The internal variable $m\colon\R_+\to\R_+$ is a function of time which does not depend on the observer state (see \cref{rem:what-m} below).
	\end{itemize}
	In \eqref{eq:rule}, for $\tau\in[t_k,t)$, $P\zhat(\tau)$ is the value of the monitored signal along the trajectory on which the control is frozen at $K\zhat(t_k)$. The instant $t_{k+1}$ is the first time at which this frozen trajectory leaves the set allowed by the threshold. The triggering instants and the trajectory are therefore constructed together, interval by interval. This is made precise in \cref{def:solution}, and \cref{thm:main-wp} below shows that this construction is well defined.
	
	\begin{remark}[Why a positive internal variable?]\label{rem:why-m}
		The static relative rule, which corresponds to $m\equiv0$, may fail to define a sequence of events. Consider an event at which $\zhat(t_k)=0$, as happens at $t_0=0$ when the observer is initialized at $\zhat_0=0$. Then $P\zhat(t_k)=0$, and both sides of the inequality in \eqref{eq:rule} vanish at $\tau=t_k$. Therefore, if $\gamma \in [0,1)$ and $\normQ{ P \zhat(\tau)}\neq 0$ for any $\tau > t_k$, the inequality $ \normQ{P\zhat(\tau)-P\zhat(t_k)}\leq \gamma  \normQ{P\zhat(\tau)}$ cannot be satisfied for any $\tau > t_k$, and the set in \eqref{eq:rule} is empty, thus making the triggering rule irrelevant. 
		In finite dimension, this lack of event separation under output feedback was pointed out in \cite{BorgersHeemels2014}. The internal variable prevents it, since the threshold is then positive at every event. We thus equip the triggering rule with a positive internal variable, in the spirit of the dynamic triggering mechanisms of \cite{Girard2015Dynamic}.
	\end{remark}
	
	\begin{remark}[On what may the internal variable depend?]\label{rem:what-m}
		Since the estimation error $e$ evolves independently of the triggering mechanism, the internal variable $m$ may depend on the measurable output residual $Ce=C\zhat-y$. Typical choices, studied in \cref{sec:thresholds}, include, for $t\ge0$,
		\begin{itemize}
			\item $m(t)=m_0e^{-\mu t}$ for some $m_0 >0$ and $\mu >0$,
			\item $m$ satisfying $~\dot m=-\mu m+\norm{Ce}_{\Ycal}~$ with $m(0) = m_0$ for some $m_0 >0$ and $\mu >0$,
			\item  $m(t)=m_0e^{-\mu t}+\sup_{0\le s\le t}\{e^{-\nu(t-s)}\normH{LCe(s)}\}$ for some $m_0 >0$, $\mu>0$ and $\nu >0$.
		\end{itemize}
		The last expression dominates the output residual over unit time windows with a constant that does not depend on the data. This is the key to the dwell-times of \cref{thm:main-dw}.
	\end{remark}
	
	To prove the well-posedness of event-triggered control system \eqref{eq:obs-err}--\eqref{eq:threshold} we assume that 
	\begin{alignat}{2}
		&\tag{H1}\quad& &m\in C^0(\R_+;\R_+),\label{H1}\\
		&\tag{H2}\quad& &m(t)>0\quad\text{for every }t\ge0.\label{H2}
	\end{alignat}
	To show exponential stability, the internal variable is required to decay exponentially, i.e.
	\begin{equation}\tag{H3}\label{H3}
		\exists\,C_m\ge0,\ \exists\,\mu>0,\quad m(t)\le C_me^{-\mu t}\quad\text{for every }t\ge0.
	\end{equation}
	
	\begin{remark}[Why the two terms of the threshold in \eqref{eq:threshold} are kept apart?]\label{rem:gamma-vs-m}
		Formally, \eqref{eq:rule} involves a single threshold $\Theta$, and one may wish to absorb $\gamma\normQ{P\zhat}$ into $m$. We keep the two terms apart because they obey different hypotheses and play different roles.
		\begin{itemize}[leftmargin=*]
			\item Hypotheses \eqref{H1}, \eqref{H2} and \eqref{H3} can be placed on $m$ but not on $\Theta$. Indeed, positivity fails where $P\zhat$ vanishes, and a decay assumption on $\Theta$ would presuppose the decay of $\zhat$, which is what we want to prove.
			\item In the stability analysis, $\gamma\normQ{P\zhat}$ is a feedback term that must be absorbed (small-gain condition), whereas $m$ is a forcing term that only costs its decay.
			\item In the dwell-time analysis, the roles are reversed. $\gamma\normQ{P\zhat}$ makes the threshold comparable to the monitored signal $P\zhat$, whereas $m$ keeps the threshold away from zero and dominates the output residual $LCe$, which $\gamma\normQ{P\zhat}$ cannot do since the dynamics of $e$ does not depend on $\zhat$.
		\end{itemize}
	\end{remark}
	
	Endogenous terms other than $\gamma\normQ{P\zhat}$, depending for instance on the last sampled state $\zhat(t_k)$, could also be considered. The well-posedness theory of \cref{sec:main-wp} covers them (\cref{rem:general-threshold}), but the stability and dwell-time analyses of this paper are specific to the relative term $\gamma\normQ{P\zhat}$.
	
	\paragraph{Event sequences.} The sequence $(t_k)\coloneqq(t_k)_{k\in\mathbb N}$ generated by \eqref{eq:rule} depends on the initial datum $(\zhat_0,e_0)$, on the system and on the design parameters.\\ 
	Under \eqref{H1}-\eqref{H2}, it is strictly increasing with $t_0=0$, and either infinite, or finite ($t_0<\dots<t_N$), in which case we set $t_{N+1}\coloneqq\infty$. An index $k$ is admissible if $t_k<\infty$.\\ 
	For every admissible $k$, $[t_k,t_{k+1})$ is an inter-event interval, and all statements involving $k$ are understood for admissible indices only. We set
	\[
	T^*\coloneqq\lim_{k\to\infty}t_k\in(0,\infty]\ \ \text{if the sequence is infinite},\qquad T^*\coloneqq\infty\ \ \text{otherwise}.
	\]
	The intervals $[t_k,t_{k+1})$, $k$ admissible, are disjoint and their union is $[0,T^*)$.\\ 
	For $s\in[0,T^*)$, $k(s)$ denotes the admissible index such that $s\in[t_{k(s)},t_{k(s)+1})$.
	
	\subsection{Main results}\label{sec:main}
	
	Our first result states that the event-triggered closed loop is well defined and free of Zeno behavior. 		
	\begin{theorem}[Well-posedness and no Zeno]\label{thm:main-wp}
		Let \cref{A:gen,A:BK,A:CL} hold, let $\gamma\ge0$ and let $m$ satisfy \eqref{H1}--\eqref{H2}.
		\begin{enumerate}[label=\textup{(\roman*)}]
			\item \textup{(Existence and uniqueness.)} For every initial datum $(\zhat_0,e_0)\in\Hcal\times\Hcal$, the rule \eqref{eq:rule}--\eqref{eq:threshold} generates a unique solution of \eqref{eq:obs-err} (in the sense of \cref{def:solution}). It consists of an event sequence $(t_k)$, a state estimation error $e\in C([0,T^*);\Hcal)$ and of an observer trajectory $\zhat\in C([0,T^*);\Hcal)$, which is the mild solution of \eqref{eq:obs-err} on each inter-event interval.
			\item \textup{(Exclusion of Zeno behavior.)} Moreover, the events do not accumulate in finite time, i.e. $T^*=\infty$.
		\end{enumerate}
	\end{theorem}
	
	No stability assumption is needed, and the result holds in particular when the threshold reduces to the internal variable ($\gamma=0$). It is proved in \cref{sec:wp,sec:nozeno}.
	
	\begin{theorem}[Exponential stability]\label{thm:stab}
		Let \cref{A:gen,A:BK,A:Kstab,A:CL,A:Lstab} hold, let $m$ satisfy the hypothesis \eqref{H1}-\eqref{H2}-\eqref{H3}. Let $\lambda>0$ and $\gamma\ge0$ satisfy
		\begin{equation}\label{eq:smallgain}
			\lambda\le\min\{\alpha_L,\mu\},\qquad\lambda<\alpha_K,\qquad 
			M_K\norm{D}\norm{P}\gamma<\alpha_K-\lambda .
		\end{equation}
		Then there exists $C_\lambda>0$, independent of the initial datum, such that, for every $(\zhat_0,e_0)\in\Hcal\times\Hcal$, the solution $\zhat$ of  \eqref{eq:obs-err}--\eqref{eq:rule}--\eqref{eq:threshold} satisfies
		\begin{equation}\label{eq:stab}
			\normH{\zhat(t)}\le C_\lambda\big(\normH{\zhat_0}+\normH{e_0}+C_m\big)e^{-\lambda t},\qquad t\ge0,
		\end{equation}
		and consequently
		\begin{equation}\label{eq:stab-full}
			\normH{z(t)}+\normH{\zhat(t)}\le\big(2C_\lambda+M_L\big)\big(\normH{\zhat_0}+\normH{e_0}+C_m\big)e^{-\lambda t},\qquad t\ge0.
		\end{equation}
	\end{theorem}
	
	\Cref{thm:stab} is proved in \cref{sec:stab}, where $C_\lambda$ is given explicitly (see \eqref{eq:stab_C}).  We next turn to dwell-times. A number $\tau_0>0$ is a dwell-time if $t_{k+1}-t_k\ge\tau_0$ for every admissible $k$ with $t_{k+1}<\infty$. It is independent of the initial datum if it holds for every $(\zhat_0,e_0)\in\Hcal\times\Hcal$ (see \cref{def:dwell}). By the discussion of \cref{sec:intro}, the existence of a dwell-time independent of the data requires to control the dynamics of the monitored signal $P\zhat(\cdot)$ over short times. This is possible when the sampled dynamics are bounded, in the following sense.
	\begin{assumption}[Bounded sampled dynamics]\label{A:BSD}
		There exists $\Acal_P\in\Lcal(\HP)$ such that
		\begin{equation*}
			P e^{t\Acal}=e^{t\Acal_P}P \quad \text{for every } t\ge0.
		\end{equation*}
	\end{assumption}
	\Cref{A:BSD} holds with $\Acal_P=\Acal$ when $\mathcal H_P=\mathcal H$, $P=I$ and $\Acal$ is bounded, in particular in finite dimension. It also holds when $P$ is the projection onto a finite-dimensional invariant subspace containing the unstable modes of $\Acal$, thereby requiring only theses modes to be monitored (\cref{cor:dw-spectral}). This covers parabolic equations with finitely many unstable modes (\cref{sec:examples}).
	
	We will also require additional assumptions on the internal variable $m$, namely
	\begin{alignat}{2}
		&\tag{H4}\quad& &\exists\,C_e\ge0,\quad\sup_{s\in[\max(t-1,0),\,t]}\normH{LCe(s)}\le C_e\,m(t)\quad\text{for every }t\ge0,\label{H4}\\
		&\tag{H5}\quad& &\exists\,c_m>0,\quad m(t)\ge c_m\,e^{-\mu t}\quad\text{for every }t\ge0.\label{H5}
	\end{alignat}

	\begin{theorem}[Dwell-time for bounded sampled dynamics]\label{thm:dw-bounded}
		Let \cref{A:gen,A:BK,A:CL,A:BSD} hold and consider the threshold $\Theta=\gamma\normQ{P \zhat}+m$, with $m$ satisfying \eqref{H1}--\eqref{H2}. For every $(\zhat_0,e_0)\in\Hcal\times\Hcal$ and every $t\ge0$, the solution $\zhat$ of  \eqref{eq:obs-err}--\eqref{eq:rule}--\eqref{eq:threshold} satisfies the following.
		\begin{enumerate}[label=\textup{(\roman*)}]
			\item If $\gamma>0$ and $m$ satisfies \eqref{H4}, then $t_{k+1}-t_k\ge\tau_0$ for every admissible $k$, with
			\begin{equation}\label{eq:tau-i}
				\tau_0=\min\Big\{1,\ \frac{e^{-\norm{\Acal_P}}}{(\norm{\Acal_P}+\norm{PD})(1+\gamma^{-1})+\norm{P}C_e}\Big\}.
			\end{equation}
			\item If $\gamma=0$, \cref{A:Kstab,A:Lstab} hold, and $m$ satisfies \eqref{H3} and \eqref{H5} with $\mu\le\alpha_L$, $\mu<\alpha_K$, then $t_{k+1}-t_k\ge\tau_0$ for every admissible $k$, with
			\begin{equation}\label{eq:tau-ii}
				\tau_0=\min\Big\{1,\ \frac{c_m\,e^{-\mu-\norm{\Acal_P}}}{\norm{P}(\norm{\Acal_P}+\norm{PD})\,C_\mu^*[\zhat]+\norm{PLC}M_L\normH{e_0}}\Big\},
			\end{equation}
			where 
			\begin{equation}\label{defcb}
				C_\mu^*[\zhat]=\sup_{t\ge0}\{e^{\mu t}\normH{\zhat(t)}\}
			\end{equation}
			(This quantity is finite by \cref{thm:stab}, see also the more precise estimate \eqref{eq:rate} with $\gamma=0$.)
		\end{enumerate}
		In case \textup{(i)}, $\tau_0$ depends on the initial datum only through $C_e$ (which is given in \eqref{H4}, and computed on the examples of \cref{prop:examples}).\\
		In case \textup{(ii)}, $\tau_0$ depends on the initial datum only through $C_\mu^*[\zhat]c_m^{-1}$ and $\normH{e_0}c_m^{-1}$ ($c_m$ is given in \eqref{H5}).
	\end{theorem}
	
	As a complementary result, we also state the following negative result on dwell-time when $\Acal$ is not bounded.
	
	\begin{theorem}[No Dwell-time for some unbounded generators]\label{thm:main-dw}
		There exists a system satisfying \cref{A:gen,A:BK,A:Kstab,A:CL,A:Lstab}, with $\Acal$ not bounded, for which no triggering rule admits both a dwell-time independent of the initial datum and bounded trajectories.
	\end{theorem}
	
	\begin{remark}\label{rem:app_results}
		\cref{thm:main-wp,thm:stab,thm:dw-bounded}
		apply to finite-dimensional systems and to parabolic equations, such as heat or reaction--diffusion equations, when only their finitely many unstable modes are monitored (\cref{cor:dw-spectral}). While \cref{thm:main-wp,thm:stab} still apply for wave or Schr\"odinger equations stabilized by an in-domain damping \cite{BaudouinMarxTarbouriech2019,KoudohodeBaudouinTarbouriech2022,BaudouinErvedoza2025}, \cref{thm:dw-bounded} does not, since for such equations, we do not assume the state fully monitored (see \cref{sec:num-wave}).
		In this case, the inter-event times may depend on the frequency content of the data, and \cref{thm:main-dw} shows that this cannot always be avoided. All these cases are illustrated in \cref{sec:examples}.
	\end{remark}
	
	\Cref{thm:dw-bounded,thm:main-dw} are proved in \cref{sec:dwell}.
	\Cref{thm:main-dw} holds for any rule deciding the events, the control being held constant between events (\cref{def:scheme,thm:necessity}). It shows that the restriction to bounded sampled dynamics is not a limitation of our triggering rule, but is intrinsic to some systems. Its argument is the non-uniform continuity of $C_0$-semigroups with unbounded generators (\cref{lem:nonunif}).
	
	\Cref{sec:dwell} also contains more specific results. When the threshold reduces to the internal variable ($\gamma=0$), a dwell-time independent of the data requires to normalize $m$ by the size $\normH{\zhat_0}+\normH{e_0}$ of the data (\cref{cor:dw-normalised,prop:gamma0-nec}). 
	
	\Cref{sec:thresholds} also provides a description of admissible internal variables $m$ and the computations of the corresponding coefficient $C_e$ in \eqref{H4}, which appears explicitly in the dwell-time \eqref{eq:tau-i}. In particular, we will check that depending on the choice of the internal variable $m$, $C_e$ may depend on the initial datum $e_0$ or not, see \cref{prop:examples} and \cref{tab:proposition}.
	
	\begin{remark}[Energy space]\label{rem:energy_space}
		In the work of Baudouin and Ervedoza \cite{BaudouinErvedoza2025}, well-posedness is
		established for initial data in $\Dcal(\Acal)$. Here, well-posedness,
		exclusion of Zeno behavior and exponential decay are established
		directly in $\Hcal\times\Hcal$, without assuming that the initial
		data belong to $\Dcal(\Acal)\times \Dcal(\Acal)$. 
	\end{remark}
	
	\subsection{Related references}
	The seminal work \cite{Tabuada2007} introduced a modern formulation of event-triggered control for nonlinear finite-dimensional systems.
	It highlighted state-dependent triggering mechanisms guaranteeing closed-loop stability while reducing communication load. Subsequent developments extended the framework to distributed and multi-agent systems \cite{DimarogonasFrazzoliJohansson2012}, while the survey paper \cite{HeemelsJohanssonTabuada2012} provided a systematic treatment of stability, Zeno behavior and robustness issues.
	
	In a nutshell, most of the related existing literature concerning linear systems are governed by ordinary differential equations of the form
	\begin{equation*}
		\dot{x}(t) = Ax(t) + Bu(t),\quad t\geq 0,
	\end{equation*}
	where $A$ and $B$ are matrices and $x$ and $u$ are respectively the state and the control-input vector. 
	In this setting, well-posedness and stability of the system involve Lyapunov-based arguments, norm equivalence and compactness properties.
	In the event-triggered framework, the control input is updated only at specific instants $\left(t_k\right)_{k\in \mathbb N}$ determined by a triggering condition that typically involves, in the most simple cases, a comparison between the sampling-induced deviation and the state norm, for instance implying that
	$\lVert x(t_k)-x(t)\rVert^2 < \gamma \lVert x(t)\rVert^2$, for all $t\in[t_k,t_{k+1})$,	where $\gamma \in (0,1)$ is a design parameter.
	Under suitable assumptions, such rule ensures exponential stability and guarantee strictly positive inter-event times (bounded by below by a dwell-time), the latter property excluding Zeno phenomenon. The introduction of dynamic variables in the triggering condition has further enhanced these results by providing additional flexibility \cite{Girard2015Dynamic}, and facilitating the exclusion of the Zeno behavior.
	
	In the finite-dimensional setting, the theory has reached a certain maturity.
	Static triggering rules were first studied in foundational works such as \cite{Tabuada2007,HeemelsJohanssonTabuada2012}, while periodic and self-triggered variants were developed to reduce monitoring effort and communication burden \cite{DonkersHeemels2012,HeemelsDonkersTeel2013}. More recently, dynamic triggering mechanisms have been shown to offer additional degrees of freedom \cite{Girard2015Dynamic, YiJohanssonDimarogonas2017}, enabling guarantees on the absence of Zeno behavior while improving performance and reducing communication demands. Lyapunov-based methods have proven particularly successful, leading to conditions for exponential stability and explicit bounds on inter-event times. For nonlinear systems, a general framework was provided in \cite{PostoyanTabuadaNesicAnta2015}.\\
	
	However, many physical systems of interest are inherently infinite-dimensional. Applications such as fluid dynamics, heat transfer, vibrating structures and delay systems are naturally modeled by PDE or functional differential equations, leading to state spaces of infinite dimension. This motivates the extension of event-triggered control ideas to the infinite-dimensional framework. Such systems are typically modeled in Hilbert (or Banach) spaces as
	\begin{equation*}
		\dot{z}(t) = \Acal z(t) + \mathcal{B}u(t),\quad t\geq 0,
	\end{equation*}
	where $\mathcal{\Acal}$ is an unbounded operator generating a strongly continuous semigroup. 
	
	The transition from finite to infinite dimensions is far from trivial. New challenges arise. The operator $\mathcal B$ involved may be unbounded (and thus may require admissibility conditions), the state space is no longer locally compact and the spectral properties are more complex. Therefore, the infinite-dimensional setting imposes additional conditions on the operators involved and requires more delicate estimates. Moreover, the Lyapunov-based approach must be carefully adapted to handle the semigroup setting. Accordingly, the definition of event-triggering mechanisms must respect the nature of the infinite-dimensional state.
	
	Stabilization and observation of PDE systems have been extensively 
	studied in the literature, with comprehensive treatments available in monographs such as \cite{CurtainZwart1995,Pazy1983,TucsnakWeiss2009} or 
	foundational articles including \cite{LasieckaTriggiani2000,Staffans2005,Weiss1989}. 
	In contrast, the development of event-triggered mechanisms for such systems remains 
	comparatively recent and far less complete.
	Indeed, ensuring well-posedness of the closed-loop system and guaranteeing positive minimum inter-event times require refined arguments that are not direct extensions of the finite-dimensional case. The dynamic triggering mechanism, which has proven effective in the finite-dimensional case, offers particular promise in the infinite-dimensional setting \cite{EspitiaGirardMarchandPrieur2016,KoudohodeBaudouinTarbouriech2022} as it can compensate for the lack of local compactness and provide the necessary regularity to exclude Zeno behavior. 
	
	In many practical situations, full-state measurements are not available, and stabilization must rely on output feedback. In the finite-dimensional setting, observer-based event-triggered control has been investigated \cite{EtienneDiGennaroBarbot2015,ZhangFeng2014}. These works typically leverage Lyapunov analysis to derive triggering conditions based on state estimation errors ensuring both stability and resource efficiency.\\ These strategies have been successfully extended to some class of infinite-dimensional systems governed by PDE. Backstepping-based boundary observers have proven effective in achieving exponential convergence for hyperbolic \cite{KrsticSmyshlyaev2008} and parabolic \cite{SmyshlyaevKrstic2005} equations using only boundary measurements. More recently, event-triggered observer designs have been developed for specific classes of PDE systems, including one-dimensional reaction-diffusion equations \cite{RathnayakeDiagneEspitiaKarafyllis2021} and 2x2 hyperbolic systems \cite{EspitiaPrieurTarbouriechGirard2020}, addressing both spatial dynamics and communication constraints. Despite these advances, the extension of observer-based event-triggered control to broader classes of PDE systems remains largely unexplored.\\
	
	The present work takes its root in the possibility of extending the contributions of \cite{BaudouinErvedoza2025} that considers the possibility of event-triggering the closed-loop system
	\begin{equation}
		\begin{cases}
			\dot{z}(t)=\Acal z(t)+Bu(t),&\quad t\geq 0,\\
			u(t)=Kz(t),&\quad t\geq 0,\\
			z(0)=z_0,&
		\end{cases}\label{abstract_PDE}
	\end{equation}
	where $z_0 \in \Hcal$, which is exponentially stable in $\Hcal$ by \cref{A:Kstab}.
	Indeed, in \cite{BaudouinErvedoza2025}, the authors have established the exponential stability of control system \eqref{abstract_PDE} when the control is based on an observation operator $C$ and subjected to an aperiodic sample-and-hold mechanism. Specifically, they showed that the event-triggered control system
	\begin{equation}
		\begin{cases}
			\dot{z}(t)=\Acal z(t)+Bu(t),&\quad t\geq 0,\\
			u(t)=K'y(t_k),&\quad t\in[t_k,t_{k+1}),\\
			y(t)=Cz(t),&\quad t\geq 0,\\
			z(0)=z_0,&
		\end{cases}\label{system_trigger_PDE}
	\end{equation}
	with $K'\in\Lcal(\Ycal,\Ucal)$, subjected to the event-triggering mechanism given by
	\begin{equation}\label{law_trigger_PDE}
		t_{k+1} \coloneqq\sup\Big\{ t>t_k, \quad \lVert Cz(t) -  Cz(t_k)\rVert_{\mathcal{Y}} \leq\gamma  \lVert z(t)\rVert_{\mathcal{H}}  \Big\}, \quad \gamma > 0,
	\end{equation}
	is well-posed for $z_0\in \Dcal(\Acal)$ and exponentially stable in $\Hcal$ under \cref{A:gen} and the assumptions that $\Acal$ is skew-adjoint, $B$ and $C$ are bounded, $\Acal+BK'C$ generates an exponentially stable $C_0$-semigroup on $\Hcal$ and that the tuning parameter $\gamma$ is chosen small enough. We mention that the continuous system \eqref{abstract_PDE} and the event-triggered system \eqref{system_trigger_PDE}-\eqref{law_trigger_PDE} each have a unique solution which is classical if $z_0 \in \Dcal(\Acal)$.
	
	One should note that our system of interest \eqref{eq:system} is derived from \eqref{abstract_PDE} incorporating a Luenberger observer based on an observer operator $L$, and by implementing control updates at discrete event-triggered instants. Our upcoming analysis builds on the system \eqref{system_trigger_PDE}-\eqref{law_trigger_PDE}, which already follows an event-triggered framework based on observation, but without incorporating an observer.

	The closed loop \eqref{system_trigger_PDE} is the particular case of \eqref{eq:system} in which the observer is initialized at the true state ($e_0=0$, so that $\zhat=z$), the gain is $K=K'C$, and the monitored signal is $P\zhat=Cz$ ($P=C$, $\HP=\Ycal$, $D=BK'$). In that case $e\equiv0$, and \cref{A:Lstab} plays no role in our stability argument. Our rule \eqref{eq:rule}-\eqref{eq:threshold} then differs from \eqref{law_trigger_PDE} in two respects. Its relative term is $\gamma\norm{Cz(t)}_{\Ycal}$, which is available from the measurement, instead of $\gamma\normH{z(t)}$, and it contains a positive internal variable. With respect to \cite{BaudouinErvedoza2025}, in the present article, we consider  a dynamic output feedback through an observer, we drop the skew-adjointness of the generator and the regularity of the data, we exclude Zeno behavior without any explicit lower bound on the inter-event times, and we study dwell-times.
	
	\subsection{Organization of the paper}
	
	\Cref{sec:prelim} collects preliminary results, among which the increment identity \eqref{eq:incr-identity}, on which all dwell-time estimates rest, and the non-uniform continuity of $C_0$-semigroups (\cref{lem:nonunif}), which is the source of the obstructions. \Cref{sec:main-wp} proves well-posedness and the absence of Zeno behavior (\cref{thm:main-wp}). \Cref{sec:stab} establishes exponential stability (\cref{thm:stab}). \Cref{sec:dwell} is devoted to dwell-times (\cref{thm:dw-bounded,thm:main-dw}). \Cref{sec:thresholds} describes admissible internal variables.
	\Cref{sec:examples} applies the results to several classes of systems and illustrates them numerically.
	\cref{sec:conclusion} concludes. 
	
	\section{Preliminaries}\label{sec:prelim}
	
	In this preliminary section, we state some classical and intermediate results that will be used throughout the paper. 
	
	\paragraph{Mild solutions and elementary estimates.} Throughout, integrals of $\Hcal$-valued functions are Bochner integrals. For $T\in(0,\infty]$, $v_0  \in\Hcal$ and $f\in L^1_{\mathrm{loc}}([0,T);\Hcal)$, the mild solution of $\dot v=\Acal v+f$, $v(0)=v_0  $, on $[0,T)$ is given by the Duhamel formula (see \cite[Definition 4.2.3]{Pazy1983})
	\begin{equation}\label{eq:mild}
		v(t)=e^{t\Acal}v_0  +\int_0^te^{(t-s)\Acal}f(s)\,ds,\qquad t\in[0,T).
	\end{equation}
	
	In the stability analysis, the observer will be compared with the stable closed loop generated by ${\Acal+BK}$, while its equation involves $\Acal$. The following obvious lemma allows to change the reference generator in a mild formulation.
	
	\begin{lemma}[Change of reference generator]\label{lem:change-gen}
		Let $R\in\Lcal(\Hcal)$.
		Let $T\in(0,\infty]$, $v_0  \in\Hcal$, $f\in L^1_{\mathrm{loc}}([0,T);\Hcal)$, and let $v$ be the mild solution \eqref{eq:mild} of $\dot v=\Acal v+f$, $v(0)=v_0  $. Then $v$ is also the mild solution of $\dot v=(\Acal+R)v+(f-Rv)$, $v(0)=v_0  $, that is,
		\[
		v(t) = e^{t \Acal}v_0  + \int_0^te^{(t-s)\Acal} f(s)\,ds=e^{t(\Acal+R)}v_0  +\int_0^te^{(t-s)(\Acal+R)}\big(f(s)-Rv(s)\big)\,ds,\qquad t\in[0,T).
		\]
	\end{lemma}
	
	We will also use several times the following elementary computation.
	
	\begin{lemma}[Convolution of exponentials]\label{lem:conv-exp}
		Let $\alpha>0$, $0\le\lambda<\alpha$ and $\nu\ge\lambda$. Then
		\[
		e^{\lambda t}\int_0^te^{-\alpha(t-s)}e^{-\nu s}\,ds\le\frac1{\alpha-\lambda}\quad \text{for every } t\ge0.
		\]
	\end{lemma}
	\begin{proof}
		The left-hand side equals $\displaystyle\int_0^te^{-(\alpha-\lambda)(t-s)}e^{-(\nu-\lambda)s}\,ds\le\int_0^te^{-(\alpha-\lambda)(t-s)}\,ds\le(\alpha-\lambda)^{-1}$.
	\end{proof}
	
	\paragraph{The estimation error.} We have the following intermediary result regarding the estimation error.
	
	\begin{lemma}[Error dynamics]\label{lem:error}
		Let \cref{A:gen,A:CL,A:Lstab} hold. For every $e_0\in\Hcal$, the error $e(t)=e^{t({\Acal+LC})}e_0$ is independent of the triggering instants and, for every $t,h\ge0$,
		\begin{gather}
			\normH{e(t)}\le M_Le^{-\alpha_Lt}\normH{e_0},\qquad
			\normH{LCe(t)}\le\norm{LC}M_Le^{-\alpha_Lt}\normH{e_0},\label{eq:err}\\
			\normH{e(t+h)}\le M_Le^{-\alpha_L h}\normH{e(t)} .\notag
		\end{gather}
	\end{lemma}
	\begin{proof}
		The error equation in \eqref{eq:obs-err} does not involve the control, so $e(t)=e^{t({\Acal+LC})}e_0$. The first two bounds follow from \cref{A:Lstab} and $LC\in\Lcal(\Hcal)$. The third comes from the semigroup property $e(t+h)=e^{h({\Acal+LC})}e(t)$ together with \cref{A:Lstab}.
	\end{proof}
	
	\paragraph{The increment identity.} The time-shifted Duhamel formula \eqref{eq:mild}, together with
	the bound \eqref{eq:growth} yields the following identity.
	
	\begin{lemma}[Increment identity]\label{lem:increment}
		Let \cref{A:gen,A:BK,A:CL} hold, let $s_0\ge0$, $v\in\Hcal$, and let $w$ be the mild solution of $\dot w=\Acal w+BKv+LCe$ on $[s_0,\infty)$ with $w(s_0)=v$, where $e(t)=e^{t({\Acal+LC})}e_0$. Then, for every $h\ge0$,
		\begin{equation}\label{eq:incr-identity}
			w(s_0+h)-v=\big(e^{h\Acal}-I\big)v+\int_0^he^{(h-s)\Acal}\left(BKv + LCe(s_0+s)\right)\,ds .
		\end{equation}
	\end{lemma}
	
	\begin{proof}
		\eqref{eq:incr-identity} is the mild formula \eqref{eq:mild} for the solution $w$ to $\dot w=\Acal w+BKv+LCe$ on $[s_0,s_0+h]$ with $w(s_0)=v$. 
	\end{proof}
	
	As an easy consequence of the Taylor expansion of the exponential and of the triangular inequality for norms, we have the following useful property.
	\begin{lemma}[Increments of a bounded generator]\label{lem:bounded-incr}
		Let $X$ be a Banach space and $G\in\Lcal(X)$. Then, it holds
		\begin{equation}\label{eq:inc}
			\norm{(e^{hG}-I)x}\le h\norm{G}e^{h\norm{G}}\norm{x} \quad \text{for every } h\ge0 \text{ and }x\in X.
		\end{equation}
	\end{lemma}
	
	One will see that all dwell-time estimates of this work rest on the  identity \eqref{eq:inc}.
	
	\paragraph{Non-uniform continuity of $C_0$-semigroups.}
	The following result is the quantitative form of the classical fact that a $C_0$-semigroup is continuous at $0$ in operator norm if and only if its generator is bounded (see  \cite[Corollary II.1.5]{EngelNagel2000}).
	\begin{lemma}\label{lem:nonunif}
		Let \cref{A:gen} hold and assume that 
		$\Acal$ is not bounded, i.e. that $\mathcal{D}(\Acal) \neq \Hcal$. Then for every $\delta>0$ and $\varepsilon\in(0,1)$ there exist $h\in(0,\delta]$ and $z\in\Hcal$ with $\normH{z}=1$ such that
		\begin{equation*}
			\normH{(e^{h\Acal}-I)z}\ge1-\varepsilon.
		\end{equation*}
	\end{lemma}
	\begin{proof}
		Since $\Acal$ is unbounded, the semigroup $(e^{t\Acal})_{t\geq 0}$ is not continuous in operator norm at $t=0$ (see \cite[Corollary II.1.5]{EngelNagel2000}). Moreover, by the classical zero–one law, setting
		\[
		L:=\limsup_{h\to0^+}\|e^{h\Acal}-I\|,
		\]
		the identity $2(e^{h\Acal}-I)=e^{2h\Acal}-I-(e^{h\Acal}-I)^2$ leads to $2L\le L+L^2$, and hence $L=0$ or $L\ge1$. Since $\Acal$ is unbounded, we necessarily have $L\ge1$, which proves Lemma \ref{lem:nonunif}.
	\end{proof}
	
	\section{Well-posedness and exclusion of Zeno behavior}\label{sec:main-wp}
	
	In this section, we assume that \cref{A:gen,A:BK,A:CL} hold, that $\gamma\ge0$, and that $m$ satisfies \eqref{H1}--\eqref{H2}. We give a precise definition of a solution of the event-triggered system \eqref{eq:obs-err}--\eqref{eq:threshold} introduced in \cref{sec:intro-rule}, and we prove \cref{thm:main-wp}.
	
	\subsection{Frozen continuations}
	
	Let $e(t)=e^{t({\Acal+LC})}e_0$. For $s_0\ge0$ and $v\in\Hcal$, the frozen continuation $W_{s_0,v}$ from $(s_0,v)$ is the mild solution of $\dot W_{s_0,v}=\Acal W_{s_0,v}+BKv+LCe$ on $[s_0,\infty)$ with $W_{s_0,v}(s_0)=v$, namely,
	\begin{equation}\label{eq:frozen}
		W_{s_0,v}(t)\coloneqq e^{(t-s_0)\Acal}v+\int_{s_0}^te^{(t-s)\Acal}\big(BKv+LCe(s)\big)\,ds,		\qquad t\ge s_0 .
	\end{equation}
	It is clear that $W_{s_0,v}\in C([s_0,\infty);\Hcal)$. It is the trajectory of the observer that one obtains if an event occurs at time $s_0$ with sampled state $v$, and no event occurs after $s_0$.
	
	\subsection{Threshold and notion of solution}
	
	In \eqref{eq:rule}-\eqref{eq:threshold}, the value of the monitored signal $P\hat z$ after $t_k$ is not known before the next event has been determined, since it depends on the control applied on $[t_k,t_{k+1})$. The threshold is therefore evaluated on the frozen continuation. The threshold attached to an event at time $s_0$ with sampled observer state $v$ is
	\begin{equation}\label{eq:ruleA}
		\Theta_{s_0,v}(t)\coloneqq\gamma\normQ{PW_{s_0,v}(t)}+m(t),\qquad t\ge s_0 .
	\end{equation}
	By \eqref{H1} and the continuity of $W_{s_0,v}$, $\Theta_{s_0,v}\in C([s_0,\infty);\R_+)$, and by \eqref{H2}, $\Theta_{s_0,v}(t)\ge m(t)>0$ for every $t\ge s_0$. The case $\gamma=0$ gives $\Theta_{s_0,v}=m$.
	
	\begin{definition}[Solution]\label{def:solution}
		Let $(\zhat_0,e_0)\in\Hcal\times\Hcal$. By solution of the event-triggered system \eqref{eq:obs-err}--\eqref{eq:threshold}, we mean a pair $\big((t_k),\zhat\big)$, where $(t_k)$ is an event sequence (strictly increasing, $t_0=0$, finite or infinite) with associated $T^*$, and $\zhat\in C([0,T^*);\Hcal)$ with $\zhat(0)=\zhat_0$, such that for every admissible index $k$,
		\begin{align}
			&\zhat(t)=W_{t_k,\zhat(t_k)}(t) \quad \text{for every } t\in[t_k,t_{k+1}), \label{S1} \tag{S1}\\
			&t_{k+1}=\sup\big\{t>t_k,\quad \normQ{P\big(W_{t_k,\zhat(t_k)}(\tau)-\zhat(t_k)\big)}\le\Theta_{t_k,\zhat(t_k)}(\tau)\ \text{for all }\tau\in[t_k,t)\big\} \label{S2} \tag{S2}.
		\end{align}
		with $W_{s_0,v}$ defined in \eqref{eq:frozen} and for any $t\geq 0$,  $e(t)=e^{t({\Acal+LC})}e_0$. Then, we write
		\begin{equation} \label{def::theta}
			\Theta(t)\coloneqq\Theta_{t_{k(t)},\zhat(t_{k(t)})}(t), \quad \text{on }[0,T^*).
		\end{equation}
	\end{definition}

	\begin{remark}[On the notation]
		Along a solution, \eqref{S1} gives
		\[
		P W_{t_k,\zhat(t_k)}(\tau)=P\zhat(\tau) \quad \text{for }\tau\in[t_k,t_{k+1}),
		\]
		so that $\Theta(\tau)=\gamma\normQ{P\zhat(\tau)}+m(\tau)$, and \eqref{S2} is exactly the rule \eqref{eq:rule}--\eqref{eq:threshold}. The two notations play different roles.
		The function $\Theta$ of \eqref{eq:threshold} is a function of time, but it is only defined once the solution is known, and it depends on the state through $P\zhat$. The family $\Theta_{s_0,v}$ is defined before the solution is constructed, and it makes explicit that the threshold used on $[t_k,t_{k+1})$ is determined by the event $(t_k,\zhat(t_k))$, through the frozen continuation (and by $e_0$, through the output residual $LCe$).
	\end{remark}
	
	\begin{remark}[On the notion of solution]\label{rem:solution}
		Strictly speaking, the full state of the event-triggered system \eqref{eq:obs-err}--\eqref{eq:rule} is the pair $(\zhat,e)$, and a complete description of a solution should involve both the event sequence $(t_k)$ and the pair $(\zhat,e)$ as presented in \cref{thm:main-wp}.
		However, since the error dynamics $e(t)=e^{t({\Acal+LC})}e_0$ is not affected by the triggering rule nor the evolution of $\zhat$ between events, it is convenient from now on to simplify the notation and refer to a ``solution'' simply as the pair $((t_k),\zhat)$, with $e$ understood to be given implicitly by $e(t)=e^{t({\Acal+LC})}e_0$.
		
		One should keep in mind that the triggering instants are part of the solution. The object which is constructed, and which depends on the initial datum, is the pair $((t_k),\zhat)$.
	\end{remark}
	
	\subsection{Existence and uniqueness
	}\label{sec:wp}
	
	The result below gathers the technical tools for the proof of the well-posedness of system \eqref{eq:obs-err}--\eqref{eq:threshold}.
	
	\begin{proposition}[Technical construction and properties of the solution]\label{prop:wp-tech}
		Let \cref{A:gen,A:BK,A:CL} hold, let $\gamma\ge0$ and let $m$ satisfy \eqref{H1}--\eqref{H2}. Fix $(\zhat_0,e_0)\in\Hcal\times\Hcal$ and, for $(s_0,v)\in\R_+\times\Hcal$, set (as in \eqref{eq:frozen})
		\[
		W_{s_0,v}(\tau)\coloneqq e^{(\tau-s_0)\Acal}v+\int_{s_0}^{\tau}e^{(\tau-s)\Acal}\big(BKv+LCe(s)\big)\,ds,\qquad \tau\ge s_0,
		\]
		which is the mild solution of $\dot W_{s_0,v}=\Acal W_{s_0,v}+BKv+LCe$ on $[s_0,\infty)$ with $W_{s_0,v}(s_0)=v$.\\
		Then the following holds.
		\begin{enumerate}[label=\textup{(\roman*)}]
			\item \textup{(Well-defined inter-event map)} For every $(s_0,v)\in\R_+\times\Hcal$, the time
			\begin{equation*}
				\sigma(s_0,v)\coloneqq\sup \left\{t>s_0, \quad \normQ{P\big(W_{s_0,v}(\tau)-v\big)}\le\Theta_{s_0,v}(\tau), \quad \forall \tau\in[s_0,t] \right\}
			\end{equation*}
			is well defined in $(s_0,\infty]$, and, if finite, it satisfies the contact identity
			\begin{equation}\label{eq:contact-step}
				\normQ{P\big(W_{s_0,v}(\sigma)-v\big)}=\Theta_{s_0,v}(\sigma)\ge m(\sigma)>0,
				\qquad \sigma\coloneqq\sigma(s_0,v).
			\end{equation}
			\item \textup{(Recursive construction and uniqueness)} The sequence $(t_k)$ such that
			\begin{equation*}
				t_0\coloneqq0, \quad t_{k+1}\coloneqq\sigma(t_k,\zhat(t_k)),  \hbox{ and } \quad
				\zhat(t)\coloneqq W_{t_k,\zhat(t_k)}(t) \quad \text{for all } t\in[t_k,t_{k+1}),
			\end{equation*}
			defines a unique solution $((t_k),\zhat)$ of \eqref{eq:obs-err}--\eqref{eq:rule}--\eqref{eq:threshold} in the sense of \cref{def:solution}, with
			\begin{equation*}
				\zhat\in C([0,T^*);\Hcal).
			\end{equation*}
			Moreover, for every admissible $k$ with $t_{k+1}<\infty$, the contact identity holds
			\begin{equation}\label{eq:contact}
				\normQ{P\zhat(t_{k+1})-P\zhat(t_k)}=\Theta_{t_k,\zhat(t_k)}(t_{k+1})\ \ge\ m(t_{k+1})\ >0 .
			\end{equation}
			\item \textup{(Global representation)} For every $t\in[0,T^*)$,
			\begin{equation}\label{eq:global-A}
				\zhat(t)=e^{t\Acal}\zhat_0+\int_0^te^{(t-s)\Acal}\big(BK\zhat(t_{k(s)})+LCe(s)\big)\,ds .
			\end{equation}
			\item \textup{(Limit at a finite accumulation time)} If $T^*<\infty$, then the solution $\zhat(t) $ of \eqref{eq:obs-err}--\eqref{eq:rule}--\eqref{eq:threshold} converges strongly in $\Hcal$ as $t\to T^{*-}$.
		\end{enumerate}
	\end{proposition}
	
	\begin{proof}
		The proof proceeds in several steps.\\
		\ \\
		- \textit{Step 1: One step of the construction.} Fix $(s_0,v)\in\R_+\times\Hcal$ and set
		\[
		\Phi(\tau)\coloneqq\Theta_{s_0,v}(\tau)-\normQ{P\big(W_{s_0,v}(\tau)-v\big)},\qquad\tau\ge s_0 .
		\]
		$\Phi$ is continuous on $[s_0,\infty)$ (by \eqref{H1} and the continuity of $W_{s_0,v}$), and using \eqref{H2} we have 
		\begin{equation*}
			\Phi(s_0)=\Gamma_{s_0,v}(s_0)+m(s_0)\ge m(s_0)>0,
		\end{equation*}
		where $\Gamma_{s_0,v}(t)\coloneqq\gamma\normQ{PW_{s_0,v}(t)}$. Let $\mathcal T \coloneqq\{t>s_0,\quad \Phi\ge0\text{ on }[s_0,t)\}$. By continuity, $\Phi>0$ on $[s_0,s_0+\delta]$ for some $\delta>0$, so $\mathcal T \ne\emptyset$, and we set $\sigma(s_0,v)\coloneqq\sup \mathcal T \in(s_0,\infty]$. Hence, by definition of $\mathcal T$,  $\Phi\ge0$ on $[s_0,\sigma(s_0,v))$. If $\sigma\coloneqq\sigma(s_0,v)<\infty$, then $\Phi(\sigma)\ge0$ by continuity, and $\Phi(\sigma)>0$ is impossible, since continuity would give $\Phi>0$ on $[\sigma,\sigma+\delta')$ for some $\delta'>0$ and thus $\sigma+\delta'\in \mathcal T $. Hence \eqref{eq:contact-step} holds, which proves \textup{(i)}.\\
		\ \\
		- \textit{Step 2: Recursive construction.} Set $t_0\coloneqq0$, $\zhat(0)\coloneqq\zhat_0$. Assume that $t_0<\dots<t_k<\infty$ and $\zhat$ on $[0,t_k]$ have been constructed. Set
		\begin{equation*}
			t_{k+1}\coloneqq\sigma(t_k,\zhat(t_k))\in(t_k,\infty] \quad \text{and }\zhat(t)\coloneqq W_{t_k,\zhat(t_k)}(t)\quad \text{for }t\in[t_k,t_{k+1}).
		\end{equation*}
		If $t_{k+1}=\infty$ the construction stops (finite event sequence). Otherwise, set $\zhat(t_{k+1})\coloneqq W_{t_k,\zhat(t_k)}(t_{k+1})$ and continue.
		This produces an event sequence and a function $\zhat$ on $[0,T^*)$.
		On each $[t_k,t_{k+1})$, $\zhat$ coincides with the continuous function $W_{t_k,\zhat(t_k)}$, and at each finite $t_{k+1}$ the value $\zhat(t_{k+1})$ is the left limit of $\zhat$. Hence $\zhat\in C([0,T^*);\Hcal)$.
		Properties \eqref{S1}--\eqref{S2} of \cref{def:solution} hold by construction.
		Moreover, since $P\zhat(t_{k+1})=PW_{t_k,\zhat(t_k)}(t_{k+1})$, we deduce that \eqref{eq:contact} is \eqref{eq:contact-step} with $(s_0,v)=(t_k,\zhat(t_k))$.\\
		\ \\
		- \textit{Step 3: Uniqueness.} Let $((t'_k),\zhat')$ be another solution with $t'_0=t_0=0$ and $\zhat'(0)=\zhat(0)$. We show by induction that, for every $k$, either both sequences have stopped before $k$, or we have $t'_k=t_k<\infty$, $\zhat'(t_k)=\zhat(t_k)$ and $\zhat'=\zhat$ on $[0,t_k]$. This holds for $k=0$. By uniqueness of W (see \eqref{eq:frozen}),  if it holds for some $k\geqslant 0$, then \eqref{S2} for both solutions gives $t'_{k+1}=\sigma(t_k,\zhat(t_k))=t_{k+1}$, and \eqref{S1} gives $\zhat'=\zhat=W_{t_k,\zhat(t_k)}$ on $[t_k,t_{k+1})$.
		If $t_{k+1}<\infty$, continuity gives $\zhat'(t_{k+1})=\zhat(t_{k+1})$. Hence the two solutions coincide, which completes the proof of \textup{(ii)}.\\
		\ \\
		- \textit{Step 4: Global representation.} We prove \eqref{eq:global-A} on $[0,t_{k+1})\cap[0,\infty)$ by induction on $k$, together with its validity at $t_{k+1}$ when $t_{k+1}<\infty$. For $t\in[t_0,t_1)$ it is \eqref{S1} with $k=0$, and at $t_1<\infty$ it holds by the definition of $\zhat(t_1)$.\\
		If it holds at $t_k$, then for $t\in[t_k,t_{k+1})$, and at $t_{k+1}$ if finite, we have
		\begin{equation} \label{newRef1}
			\zhat(t)=e^{(t-t_k)\Acal}\Big[e^{t_k\Acal}\zhat_0+\int_0^{t_k}e^{(t_k-s)\Acal}g(s)\,ds\Big]+\int_{t_k}^te^{(t-s)\Acal}g(s)\,ds,\qquad g(s)\coloneqq BK\zhat(t_{k(s)})+LCe(s),
		\end{equation}
		which is \eqref{eq:global-A} by the semigroup property. The function $g$ is piecewise continuous with finitely many pieces on each $[0,t]$, $t<T^*$, hence strongly measurable and locally integrable on $[0,T^*)$. This proves \textup{(iii)}.\\
		\ \\
		- \textit{Step 5: Limit at a finite $T^*$.} Assume $T^*<\infty$, and let
		\begin{equation*}
			C_*\coloneqq\sup_{0\le s\le T^*}\norm{e^{s\Acal}}, \qquad M_*\coloneqq\sup_{0\le s\le T^*}\norm{e^{s(\Acal+LC)}}.
		\end{equation*}
		For $t<T^*$ let $Y(t)\coloneqq\sup_{0\le r\le t}\normH{\zhat(r)}$, which is finite and non-decreasing.\\
		Since $t_{k(s)}\le s$, from \cref{A:gen,A:BK,A:CL}  and the definition of $e$, it holds
		\begin{equation*}
			\normH{g(s)}\le\norm{BK}Y(s)+\norm{LC}M_*\normH{e_0}.
		\end{equation*}
		where $g$ is defined in \eqref{newRef1}. By \eqref{eq:global-A}, for $0\le r\le t<T^*$, we deduce
		\[
		\normH{\zhat(r)}\le C_*\normH{\zhat_0}+C_*T^*\norm{LC}M_*\normH{e_0}+C_*\norm{BK}\int_0^tY(s)\,ds,
		\]
		hence the inequality
		\begin{equation*}
			Y(t)\le a+b\int_0^tY(s)\,ds, \quad a\coloneqq C_*(\normH{\zhat_0}+T^*\norm{LC}M_*\normH{e_0}), \quad b\coloneqq C_*\norm{BK}.
		\end{equation*}
		By Gronwall's Lemma, $Y(t)\le ae^{bT^*}$ for every $t<T^*$.\\
		Thus $g\in L^\infty(0,T^*;\Hcal)$ and the integral term of \eqref{eq:global-A}, defined for $t\in[0,T^*]$, is continuous on $[0,T^*]$. Hence $\zhat(t)\to e^{T^*\Acal}\zhat_0+\displaystyle \int_0^{T^*}e^{(T^*-s)\Acal}g(s)\,ds$ as $t\to T^{*-}$, which proves \textup{(iv)}.
	\end{proof}
	
	\begin{proof}[Proof of \cref{thm:main-wp} -  \textup{(i)}]
		This is precisely the content of \cref{prop:wp-tech} - \textup{(ii)}. Indeed, the recursive construction of Step 2 yields a solution satisfying \eqref{S1}--\eqref{S2}, and by Step 3 this solution is unique. The regularity of $\zhat$ follows from Step 2 as well.
	\end{proof}
	
	\subsection{Exclusion of Zeno behavior
	}\label{sec:nozeno}
	
	\begin{proof}[Proof of \cref{thm:main-wp} - \textup{(ii)}]
		Fix an arbitrary initial datum $(\zhat_0,e_0)\in\Hcal\times\Hcal$ and let
		$((t_k),\zhat)$   
		be the corresponding maximal solution given by \cref{prop:wp-tech}, defined on $[0,T^*)$.
		Assume by contradiction that $T^*<\infty$. Then the event sequence is infinite and $t_k\uparrow T^*$. By \cref{prop:wp-tech} - (iv), $\zhat(t)$ has a strong limit as $t\to T^{*-}$. In particular it holds
		\begin{equation*}
			\normQ{P\zhat(t_{k+1})-P\zhat(t_k)}\le\norm{P}\normH{\zhat(t_{k+1})-\zhat(t_k)}\to0, \quad \text{as } k\to\infty.
		\end{equation*} 
		On the other hand, by \eqref{H1}--\eqref{H2}, $\delta\coloneqq\min_{[0,T^*]}m>0$, and the contact identity \eqref{eq:contact} gives $\normQ{P\zhat(t_{k+1})-P\zhat(t_k)}\ge m(t_{k+1})\ge\delta$ for every $k$. This brings a contradiction and proves the exclusion of Zeno behavior.
	\end{proof}

	\begin{remark}[Other thresholds]\label{rem:general-threshold}
		The proof of \cref{thm:main-wp} only uses the following two properties of the relative term 
		$\Gamma_{s_0,v}(t)=\gamma\normQ{PW_{s_0,v}(t)}$ of the threshold.
		\begin{itemize}[leftmargin=*]
			\item  For every $(s_0,v)\in\R_+\times\Hcal$, the function
			$\Gamma_{s_0,v}:[s_0,\infty)\to\R_+$ is prescribed independently
			of future triggering instants, it may depend on the fixed error
			trajectory $e(t)=e^{t({\Acal+LC})}e_0$.
			\item $\Gamma_{s_0,v} \in C([s_0,\infty);\R_+)$.
		\end{itemize}
		As a consequence, the results of \cref{thm:main-wp} remain valid if $\gamma\normQ{PW_{s_0,v}}$ is replaced by any family $(\Gamma_{s_0,v})$ with these properties, for instance a family depending on the sampled state, such as $\Gamma_{s_0,v}(t)=\gamma e^{-\hat\mu(t-s_0)}\normH{v}$.
		The positivity of the threshold at each event, and hence the separation of events, is ensured by $m$ alone. The other results of this paper use more specific properties of the relative term of the threshold.
	\end{remark}

	\begin{remark}\label{rem:no-rate}
		\begin{itemize}[leftmargin=*]
			\item In \cite{BaudouinErvedoza2025}, the exclusion of Zeno behavior was obtained under additional structural assumptions (in particular initial data in the domain of the operator and skew-adjointness of $\mathcal A$), from an explicit lower bound on $t_{k+1}-t_k$ depending on $T^*$. The argument above, rooted in (iv) of \cref{prop:wp-tech} makes this unnecessary.
			\item 		\Cref{thm:main-wp} - (ii) excludes accumulation of events at a finite time, for every $C_0$-semigroup and every threshold of the form \eqref{eq:ruleA}, but gives no information on how small the inter-event times $t_{k+1}-t_k$ may become along an infinite sequence. They may in principle tend to $0$ while $t_k\to\infty$. A positive lower bound, independent of $k$ (and possibly of the initial datum), is a strictly stronger property, studied in \cref{sec:dwell}.
		\end{itemize}
	\end{remark}
	
	\section{Stability analysis}\label{sec:stab}
	
	In this section, we prove \cref{thm:stab}. Throughout, \cref{A:gen,A:BK,A:Kstab,A:CL,A:Lstab} hold, $\gamma\ge0$, and $m$ satisfies \eqref{H1}--\eqref{H2}. By \Cref{thm:main-wp} - (ii), $T^*=\infty$, and the following corollary provides the representation of the observer $\hat z$ around the stable closed-loop generator ${\Acal+BK}$, in which the sampling error enters as a forcing term bounded by $\norm{D}\Theta$.
	
	For the stability analysis, we rewrite the global mild formulation with respect to the stabilized generator $\Acal+BK$ as below. In this
	representation, the sampling error appears as an additive perturbation controlled by the triggering threshold \eqref{eq:rule}--\eqref{eq:threshold} using the formulation extracted from \eqref{eq:obs-err} and using \eqref{eq:factor}, i.e.
	$$
	\dot{\zhat}(t)=(\Acal +BK)\zhat(t)+DP(\zhat(t_k)-\zhat(t))+LCe(t), \quad t\in[t_k,t_{k+1}).
	$$
	
	\begin{corollary}[Global representation]\label{cor:global}
		Under the assumptions of \cref{prop:wp-tech}, the solution $\zhat$ of \eqref{eq:obs-err}--\eqref{eq:rule}--\eqref{eq:threshold} satisfies  \eqref{eq:global-A} for all $t\ge0$, and thus 
		\begin{equation}\label{eq:global-K}
			\zhat(t)=e^{t({\Acal+BK})}\zhat_0+\int_0^t e^{(t-s)({\Acal+BK})}\Big(LCe(s)+D\big(P\zhat(t_{k(s)})-P\zhat(s)\big)\Big)\,ds.
		\end{equation}
		Consequently,  for all $t \geq 0$,
		\begin{multline}\label{eq:bound_zhat}
			\normH{\zhat(t)}\le M_Ke^{-\alpha_Kt}\normH{\zhat_0}+M_KM_L\norm{LC}\normH{e_0}\int_0^te^{-\alpha_K(t-s)}e^{-\alpha_Ls}\,ds\\
			+M_K\norm{D}\int_0^te^{-\alpha_K(t-s)}\Theta(s)\,ds.
		\end{multline}
	\end{corollary}
	\begin{proof}
		By \cref{thm:main-wp}, $T^*=\infty$.
		Apply \cref{lem:change-gen} to the formulation  \eqref{eq:global-A} with
		$R=BK$ and $f(s)=BK(\zhat(t_{k(s)})-\zhat(s))+LCe(s)$, and use
		\eqref{eq:factor}. The last estimate follows from \eqref{eq:global-K}, \cref{A:Kstab}, \cref{lem:error} and \eqref{S2} of \cref{def:solution}.
	\end{proof}
	
	The first two terms in the right hand side of \eqref{eq:bound_zhat} obviously decay exponentially. To be able to conclude that a solution $\zhat$ of \eqref{eq:obs-err}--\eqref{eq:rule}--\eqref{eq:threshold} decays exponentially, one should thus handle the last term in the right hand side of \eqref{eq:bound_zhat}. 
	
	\begin{lemma}
		\label{Lem-Stab}
		Assume that the assumptions of \cref{prop:wp-tech} hold. Additionally, suppose that $m$ satisfy assumption \eqref{H3} and let $\lambda >0$ with $\lambda\le\min\{\alpha_L,\mu\}$, $\lambda<\alpha_K$, $M_K \norm{D} \gamma\norm{P} < \alpha_K - \lambda$. Then  the solution $\zhat$ of \eqref{eq:obs-err}--\eqref{eq:rule}--\eqref{eq:threshold} satisfies
		\begin{multline}\label{eq:rate}
			C_\lambda^*[\zhat]\coloneqq\sup_{t\ge0}\{ e^{\lambda t}\normH{\zhat(t)}\} 	
			\\
			\le\frac{\alpha_K-\lambda}{\alpha_K-\lambda-M_K\norm{D}\norm{P}\gamma }\left(M_K\normH{\zhat_0}+\frac{M_KM_L\norm{LC}}{\alpha_K-\lambda}\normH{e_0}+M_K\norm{D} J_\lambda[m] \right), 
		\end{multline}
		where we have set 
		\begin{equation} \label{TLRema2}
			J_\lambda[m] \coloneqq \sup_{t\ge0}\left\{ e^{\lambda t}\int_0^te^{-\alpha_K(t-s)}m(s)\,ds\right\}.
		\end{equation}
	\end{lemma}
	
	\begin{remark}
		Using  \cref{lem:conv-exp} and under assumption \eqref{H3}, we have $J_\lambda[m]< \infty$ for $\lambda \leq \mu$.
	\end{remark}
	\begin{proof}
		Note that for every $s\ge0$, $\Theta(s)=\gamma\normQ{P\zhat(s)}+m(s)\le\gamma\norm{P}\normH{\zhat(s)}+m(s)$.	\\
		Taking $\lambda>0$ as in assumptions of \cref{Lem-Stab}, we derive from \eqref{eq:bound_zhat}, and using  \cref{lem:conv-exp}, that for all $t \geq 0$, 
		\begin{align*}
			& e^{\lambda t}\normH{\zhat(t)}\\
			&\le M_K\normH{\zhat_0}+\frac{M_KM_L\norm{LC}}{\alpha_K-\lambda}\normH{e_0}+M_K\norm{D} e^{\lambda t}\int_0^t e^{ - \alpha_K (t-s)}  (\gamma\norm{P}\normH{\zhat(s)}+m(s)) \, ds
			\\ 
			&  \le M_K\normH{\zhat_0}+\frac{M_KM_L\norm{LC}}{\alpha_K-\lambda}\normH{e_0}
			+\frac{M_K\norm{D} \norm{P}\gamma}{\alpha_K - \lambda} \sup_{s \in [0,t]} \{ e^{\lambda s} \normH{\zhat(s)} \} 
			+ 
			M_K\norm{D} J_\lambda[m],
		\end{align*}
		where $J_\lambda[m]$ is defined in  \eqref{TLRema2}. We deduce that, if $\frac{M_K\norm{D} \norm{P}\gamma}{\alpha_K - \lambda} <1$, then
		$$
		\left( 1 - \frac{M_K\norm{D} \norm{P}\gamma}{\alpha_K - \lambda} \right) \sup_{s \in [0,t]} \{ e^{\lambda s} \normH{\zhat(s)} \} 
		\leq 
		M_K\normH{\zhat_0}+\frac{M_KM_L\norm{LC}}{\alpha_K-\lambda}\normH{e_0}+
		M_K\norm{D} J_\lambda[m], 
		$$
		which concludes the proof of \cref{Lem-Stab}.
	\end{proof}
	
	As a corollary, we prove \cref{thm:stab} with the explicit constant
	\begin{equation}\label{eq:stab_C}
		C_\lambda\coloneqq\frac{M_K(\alpha_K-\lambda)}{\alpha_K-\lambda-M_K\norm{D}\norm{P}\gamma }\max\left\{1,\frac{M_L\norm{LC}}{\alpha_K-\lambda},\frac{\norm{D}}{\alpha_K-\lambda}\right\}.
	\end{equation}
	
	\begin{proof}[Proof of \cref{thm:stab}]
		By \eqref{H3} and \cref{lem:conv-exp}, $J_\lambda[m]\le C_m (\alpha_K-\lambda)^{-1}$, and  \cref{Lem-Stab} gives \eqref{eq:stab}. Since $z=\zhat-e$ and $\normH{e(t)}\le M_Le^{-\alpha_Lt}\normH{e_0}\le M_Le^{-\lambda t}\normH{e_0}$, \eqref{eq:stab-full} follows.
	\end{proof}
	
	\begin{remark}[Dependence on the internal variable]\label{rem:Cm}
		The constant $C_m$ of \eqref{H3} appearing  in \eqref{eq:stab} may depend on the initial
		error $e_0$. This does not affect implementability since $C_m$ enters
		the stability estimate, but is not needed to evaluate the
		triggering threshold. Implementability only requires that $m(t)$
		be computable from known quantities and available measurements.
		For the output-residual filters introduced in \cref{sec:thresholds},
		$m(t)$ is computed from $Ce=C\zhat-y$ or $LCe$, without prior
		knowledge of $e_0$, and one can take
		$C_m=m_0+c\normH{e_0}$, with $c$ independent of the initial datum.
	\end{remark}
	
	\section{Dwell-time}\label{sec:dwell}
	
	In this section, \cref{A:gen,A:BK,A:CL} hold, $\gamma\ge0$, and $m$ satisfies \eqref{H1}--\eqref{H2}. We prove \cref{thm:dw-bounded,thm:main-dw} and give more specific results. 
	\Cref{sec:dwell-bounded} deals with bounded sampled dynamics and contains the proof of \cref{thm:dw-bounded}. We then apply \cref{thm:dw-bounded} to study the case $\gamma=0$ with normalized internal variables, the finite-dimensional case and the case of a finite-dimensional unstable part. \Cref{sec:dwell-bounded} concludes  with the proof of \cref{thm:main-dw} - (i). \Cref{sec:necessity} shows that, for some systems, no triggering rule can combine a dwell-time independent of the data with bounded trajectories, which proves \cref{thm:main-dw} - (ii).
	\begin{definition}[Dwell-time]\label{def:dwell}
		Let $\Ecal\subset\Hcal\times\Hcal$ be a class of initial data. A number $\tau_0>0$ is a dwell-time over $\Ecal$ if, for every $(\zhat_0,e_0)\in\Ecal$ and every admissible $k$ with $t_{k+1}<\infty$, one has $t_{k+1}-t_k\ge\tau_0$.\\ 
		If $\Ecal$ is reduced to one datum, $\tau_0$ is a dwell-time along the trajectory.\\ 
		If $\Ecal=\Hcal\times\Hcal$, $\tau_0$ is a dwell-time independent of the initial datum.
	\end{definition}
	A dwell-time along a trajectory implies $t_k\ge k\tau_0$, hence the absence of Zeno behavior, and bounds the number of events per unit time by $1/\tau_0+1$. It is strictly stronger than \cref{thm:main-wp} - (ii) (see \cref{rem:no-rate}).
	
	\subsection{Dwell-times for bounded sampled dynamics}\label{sec:dwell-bounded}
	
	\begin{proof}[Proof of \cref{thm:dw-bounded}]
		Let $k$ be admissible with $t_{k+1}<\infty$ and $h_k=t_{k+1}-t_k\le1$ (otherwise there is nothing to prove). From  \eqref{eq:factor} ($BK = DP$), \eqref{eq:incr-identity}, and \cref{A:BSD} that defines $\Acal_P$ such that $Pe^{t\Acal} = e^{t\Acal_P}P$, we get 
		\begin{multline*}
			P\zhat(t_{k+1})-P\zhat(t_k)
			\\
			=\big(e^{(t_{k+1}-t_k)\Acal_P}-I\big)P\zhat(t_k)+\int_{t_k}^{t_{k+1}} e^{(t_{k+1} -s)\Acal_P}PDP\zhat(t_k)\,ds+\int_{t_k}^{t_{k+1}}e^{(t_{k+1}-s)\Acal_P}PLCe(s)\,ds.
		\end{multline*}
		From the definition of $\Theta$ in  \eqref{eq:threshold},  \cref{lem:bounded-incr} and \cref{prop:wp-tech}, we then obtain the following key estimate
		\begin{multline}
			\label{eq:dw-proof}
			\Theta(t_{k+1}) = \normQ{P\zhat(t_{k+1})-P\zhat(t_k)}
			\le
			h_k   e^{\| \Acal_P\| }  \left( \| \Acal_P\| + \norm{PD}\right)
			\normQ{P\zhat(t_k)}
			\\
			+ h_k   e^{\| \Acal_P\| }  \sup_{s \in [t_k,t_{k+1}]} \normQ{PLCe(s)}.
		\end{multline}
		{(i) Assume that  $\gamma >0$ in \eqref{eq:threshold} and \eqref{H4}.} From   \cref{prop:wp-tech},  $h_k\le1$, $[t_k,t_{k+1}]\subset[\max(t_{k+1}-1,0),t_{k+1}]$, and \eqref{H4} at $t=t_{k+1}$, we have 
		\begin{equation*}
			\sup_{[t_k,t_{k+1}]}\normH{LCe}\le C_em(t_{k+1})\le C_e\Theta(t_{k+1}).
		\end{equation*}
		Besides, by the triangle inequality, the lower bound $\Theta(t_{k+1})\ge\gamma\normQ{P\zhat(t_{k+1})}$ given by \eqref{eq:threshold} (recall that $\gamma>0$), and \eqref{eq:contact}, we obtain
		\begin{equation}
			\label{size-vs-Theta}
			\normQ{P\zhat(t_k)}\le\normQ{P\zhat(t_{k+1})}+\normQ{P\zhat(t_{k+1})-P\zhat(t_k)}\le\gamma^{-1}\Theta(t_{k+1})+\Theta(t_{k+1}).
		\end{equation}
		Inserting these two estimates in \eqref{eq:dw-proof} and dividing by $\Theta(t_{k+1})>0$ gives
		\[
		1\le h_ke^{\norm{\Acal_P}}\left((\norm{\Acal_P}+\norm{PD})(1+\gamma^{-1})+\norm{P}C_e\right),
		\]
		which brings \eqref{eq:tau-i}.\\
		{(ii) Assume that  $\gamma =0$ in \eqref{eq:threshold}, \eqref{H3}, \eqref{H5},  \cref{A:Kstab} and \ref{A:Lstab}.} Using \eqref{H5} and since $h_k\le1$, we have
		\begin{equation}\label{3stars}
			\Theta(t_{k+1})=m(t_{k+1})\ge c_me^{-\mu t_{k+1}}\ge c_me^{-\mu}e^{-\mu t_k}.
		\end{equation}
		Since $\mu\le\alpha_L$, $\mu<\alpha_K$, by \cref{Lem-Stab} with $\gamma=0$ and $\lambda = \min(\mu,\alpha_R)=\mu$,  
		$C_\mu^*[\zhat]\coloneqq\sup_{t\ge0}\{ e^{\mu t}\normH{\zhat(t)}\}$ is bounded 		
		and we deduce
		\begin{equation}\label{2stars}
			\normQ{P\zhat(t_k)}\le\norm{P}C_\mu^*[\zhat]e^{-\mu t_k}
		\end{equation}
		Moreover, since $\mu\le\alpha_L$, from \cref{lem:error}, we obtain 
		\begin{equation*}
			\sup_{[t_k,t_{k+1}]}\normQ{PLCe}\le\norm{PLC}M_L\normH{e_0}e^{-\mu t_k}.
		\end{equation*}
		Hence, together with \eqref{eq:dw-proof}, \eqref{3stars} and \eqref{2stars}, 
		\[
		c_me^{-\mu}e^{-\mu t_k}\le h_ke^{\norm{\Acal_P}}\left(\norm{P}(\norm{\Acal_P}+\norm{PD})C_\mu^*[\zhat]+\norm{PLC} M_L\normH{e_0}\right)e^{-\mu t_k},
		\]
		bringing \eqref{eq:tau-ii}.
	\end{proof}
	
	\begin{remark}[The role of $\gamma$]\label{rem:role-gamma}
		In case (i), the relative term $\gamma\normQ{P\zhat}$ provides the local comparison of \eqref{size-vs-Theta} between the sampled state and the threshold, so that no decay assumption on $m$ is needed. As $\gamma\to0$, \eqref{eq:tau-i} degenerates. Indeed, when $\gamma=0$, nothing in the triggering rule relates $\normQ{P\zhat(t_{k+1})}$ to $m(t_{k+1})$, and $P\zhat(t_{k+1})$ may be much larger than $m(t_{k+1})$ if $m$ has decayed faster than the trajectory.\\ 
		Case (ii) replaces the local comparison by a global one, based on the exponential decay of $\zhat$ at the exact rate $\mu$ of the lower bound \eqref{H5}. This is why the dwell-time then depends on the data. 
	\end{remark}
	
	\paragraph{Threshold reduced to the internal variable ($\gamma=0$)}
	
	By \cref{thm:dw-bounded} - (ii), a threshold reduced to $m$ yields a dwell-time along each trajectory, which depends on the data through $C^*_\mu[\zhat] c_m^{-1}$ and $\normH{e_0}c_m^{-1}$. The two results below show that this dependence can be removed by normalizing $m$ with the data, and that it cannot be removed otherwise.
	
	\begin{corollary}[Normalized internal variable, $\gamma=0$]\label{cor:dw-normalised}
		Let \cref{A:gen,A:BK,A:Kstab,A:CL,A:Lstab,A:BSD} hold and consider the threshold $\Theta=m$ with the normalized internal variable (called \textnormal{(E4)} in \cref{prop:examples} later)
		\begin{equation}
			\label{Def-m-E4}
			m(t)=m_0e^{-\mu t}\hat E_0 \quad \text{where } m_0>0, \quad 0<\mu\le\alpha_L, \quad \mu<\alpha_K,~ \hbox{ and } \hat E_0>0.
		\end{equation}
		Set
		\begin{equation}\label{defcbb}
			\bar C_\mu\coloneqq M_K\Big[1+\frac{M_L\norm{LC}}{\alpha_K-\mu}+\frac{\norm{D}m_0}{\alpha_K-\mu}\Big],
		\end{equation}
		and 
		\begin{equation}\label{eq:tau-iii}
			\tau_\star\coloneqq\min\Big\{1,\ \frac{m_0\,e^{-\mu-\norm{\Acal_P}}}{\norm{P}\left((\norm{\Acal_P}+\norm{PD})\,\bar C_\mu+\norm{LC}M_L\right)}\Big\}.
		\end{equation}
		Then, for every initial datum $(\zhat_0,e_0)$ with $\hat E_0\geq \normH{\zhat_0}+\normH{e_0}>0$, it holds
		\begin{equation}\label{eq:normalised-stab}
			\normH{\zhat(t)}\le\bar C_\mu \hat E_0 e^{-\mu t}\quad \text{for every } t\ge0,
			\quad \text{and } t_{k+1}-t_k\ge\tau_\star ~ \text{for every admissible $k$}.
		\end{equation}
		In other words, the dwell-time $\tau_\star$ is independent of the initial datum. \\
		Moreover, if $\hat E_0 = \normH{\zhat_0}+\normH{e_0}=0$, then $\zhat\equiv0$ and no event occurs.
	\end{corollary}
	
	\begin{proof}
		The normalized internal variable $m$, defined by \eqref{Def-m-E4}, satisfies \eqref{H1}--\eqref{H3}, \eqref{H5} with $c_m=m_0\hat E_0$ and the rate $\mu$. Moreover, from \cref{lem:conv-exp} and \eqref{TLRema2}, we have   $J_\mu[m]\le m_0\hat E_0(\alpha_K-\mu)^{-1}$. \Cref{Lem-Stab} with $\gamma=0$  and $\lambda=\mu$, which is allowed since $\mu\le\alpha_L$ and $\mu<\alpha_K$, gives
		\[
		C^*_\mu[\hat z]=\sup_{t\ge0}\{e^{\mu t}\normH{\zhat(t)}\}
		\le M_K\normH{\zhat_0}+\frac{M_KM_L\norm{LC}}{\alpha_K-\mu}\normH{e_0}+\frac{M_K\norm{D}m_0\hat E_0}{\alpha_K-\mu}\le\bar C_\mu \hat E_0,
		\]
		since $\normH{\zhat_0}\le \hat E_0$ and $\normH{e_0}\le \hat E_0$. This gives the first estimate in \eqref{eq:normalised-stab}.\\ 
		Besides, \Cref{thm:dw-bounded} - (ii) applies, and from \eqref{eq:tau-ii},  $\norm{PLC}\le\norm{P}\norm{LC}$, and using $c_m=m_0\hat E_0$, $C^*_\mu[\zhat]\le\bar C_\mu \hat E_0$ and $\normH{e_0}\le \hat E_0$, we deduce
		\[
		\frac{c_m\,e^{-\mu-\norm{\Acal_P}}}{\norm{P}\big[(\norm{\Acal_P}+\norm{PD})C^*_\mu[\zhat]+\norm{LC}M_L\normH{e_0}\big]}\ 
		\ge\ \frac{m_0\,e^{-\mu-\norm{\Acal_P}}}{\norm{P}\big[(\norm{\Acal_P}+\norm{PD})\bar C_\mu+\norm{LC}M_L\big]}
		\]
		so that \eqref{eq:tau-ii} is bounded below by $\tau_\star$, which is therefore also a dwell-time.
	\end{proof}
	
	\begin{proposition}[Without normalization, no dwell-time independent of the data when $\gamma=0$]\label{prop:gamma0-nec}
		Let \cref{A:gen,A:BK,A:CL,A:BSD} hold and consider the threshold $\Theta=m$, where $m$ satisfies \eqref{H1}--\eqref{H2}.
		Let $h\in\Hcal$ be such that $v\coloneqq (\Acal_P+PD)Ph\ne0$, and set $\bar m\coloneqq\max_{[0,1]}m$,
		\[
		c\coloneqq\dfrac12e^{\norm{\Acal_P}}\norm{\Acal_P}\big(\norm{\Acal_P}+\norm{PD}\big)\normQ{Ph},\qquad t_*\coloneqq\min\Big\{1,\frac{\normQ{v}}{2c}\Big\}\quad(t_*=1\text{ if }c=0).
		\]
		Then, for the initial data $(\zhat_0,e_0)=(Rh,0)$ with $R>2\bar m(t_*\normQ{v})^{-1}$, the first event satisfies
		\[
		t_1\le\frac{2\bar m}{R\normQ{v}} \to0 \quad \text{as } R\to\infty .
		\]
		In particular, no internal variable of this type yields a dwell-time independent of the initial datum.
	\end{proposition}
	
	\begin{proof}
		Since $m$ is independent of the initial state of the observer $\zhat_0$ and $\zhat_0=Rh$,  $m$ does not depend on $R$. The frozen continuation from $(0,Rh)$, meaning $W_{0,Rh}(0)=Rh$, is $W_{0,Rh}=RW_{0,h}$ by linearity. Using \eqref{eq:incr-identity} (with $e\equiv0$), 
		\[
		PW_{0,h}(t)-Ph = \big(e^{t\Acal_P}-I\big)P h+\int_0^te^{(t-s)\Acal_P}PDP h\,ds
		.
		\]
		Denoting  $v\coloneqq (\Acal_P+PD)Ph\ne0$, since
		\begin{equation*}
			PW_{0,h}(t)-Ph-tv=(e^{t\Acal_P}-I-t\Acal_P)P h+\int_0^t(e^{(t-s)\Acal_P}-I)PDP h\,ds,
		\end{equation*}
		and using
		\begin{equation*}
			\norm{e^{s\Acal_P}-I-s\Acal_P}
			\le\frac12 s^2\norm{\Acal_P}^2e^{s\norm{\Acal_P}} \quad \text{for every }s\ge0,
		\end{equation*} 
		as well as $\norm{e^{s\Acal_P}-I}\le s\norm{\Acal_P}e^{s\norm{\Acal_P}}$ (\cref{lem:bounded-incr}), we obtain for $t\in[0,1]$,
		\[
		\normQ{PW_{0,h}(t)-Ph-tv}\le\tfrac12t^2e^{\norm{\Acal_P}}\norm{\Acal_P}^2\normQ{P h}+\tfrac12t^2e^{\norm{\Acal_P}}\norm{\Acal_P}\norm{PD}\normQ{P h}=ct^2 .
		\]
		Hence we deduce (by definition of $t_*$), using a reverse triangle inequality,
		\begin{equation*}
			\normQ{PW_{0,h}(t)-Ph}\ge t\normQ v-ct^2\ge\frac12t\normQ v \quad \text{for }t\in[0,t_*].
		\end{equation*}
		Let $R>2\bar m(t_*\normQ v)^{-1}$, so that $\tau_R\coloneqq2\bar m(R\normQ v)^{-1}<t_*$. For every $t\in(\tau_R,t_*]$, it holds
		\[
		\normQ{P\big(W_{0,\zhat_0}(t)-\zhat_0\big)}=R\normQ{PW_{0,h}(t)-Ph}\ge\tfrac12Rt\normQ v>\bar m\ge m(t)=\Theta_{0,\zhat_0}(t),
		\]
		so that the condition in \eqref{S2} of \cref{def:solution} fails at $t$, and $t_1\le t$. Letting $t\downarrow\tau_R$ gives $t_1\le\tau_R$.
	\end{proof}
	
	\begin{remark}[Interpretation]\label{rem:gamma0}
		For $e_0=0$, scaling the initial datum by $R$ divides the	effective threshold by $R$ in \cref{prop:gamma0-nec}.
		When $\gamma>0$, the relative term $ \gamma\normQ{P\zhat}$ of the threshold $\Theta$ defined in \eqref{eq:threshold} scales with the data, allowing in some cases to yield a data-independent dwell-time.
		When $\gamma=0$, the internal variable must account for the size of the data. The internal variable in \eqref{Def-m-E4} does so, but requires information
		on $\normH{e_0}$.
	\end{remark}
	
	\paragraph{Dwell-times for finite-dimensional systems}
	
	\begin{corollary}[Finite dimension]\label{cor:dw-finite}
		Let $\Hcal=\R^n$ (or let $\Acal\in\Lcal(\Hcal)$), $P=I$, $D=BK$, and let \cref{A:BK,A:CL} hold.
		\begin{enumerate}[label=\textup{(\alph*)}]
			\item If $\gamma>0$ and $m$ satisfies \eqref{H1}, \eqref{H2}, \eqref{H4}, then a dwell-time is given by
			\begin{equation}\label{eq:tau-i_finite}
				\tau_0=\min\left\{1,\frac{e^{-\norm{\Acal}}}{(\norm{\Acal}+\norm{BK})(1+\gamma^{-1})+C_e}\right\}.
			\end{equation}  
			\item If $\gamma=0$, \cref{A:Kstab,A:Lstab} hold, and $m$ satisfies \eqref{H5} with $\mu\le\alpha_L$, $\mu<\alpha_K$, then a dwell-time along each trajectory is given by
			\begin{equation}\label{eq:tau-ii_finite}
				\tau_0=\min\left\{1,\ \frac{c_m\,e^{-\mu-\norm{\Acal}}}{(\norm{\Acal}+\norm{BK})\,C_\mu^*[\zhat]+\norm{LC}M_L\normH{e_0}}\right\},
			\end{equation}
			where $C_\mu^*[\zhat]$ is defined in \eqref{defcb}.
			\item If $\gamma=0$, \cref{A:Kstab,A:Lstab} hold and $m$ is the normalized internal variable \eqref{Def-m-E4} with $\mu\le\alpha_L$, $\mu<\alpha_K$, and $\hat E_0\geq \normH{\zhat_0}+\normH{e_0}>0$.
			Then a dwell-time (independent of the initial datum) is given by
			\begin{equation}\label{eq:tau-iii_finite}
				\tau_\star\coloneqq\min\left\{1,\ \frac{m_0\,e^{-\mu-\norm{\Acal}}}{(\norm{\Acal}+\norm{BK})\,\bar C_\mu+\norm{LC}M_L}\right\}
			\end{equation}
			where $\bar C_\mu$ is defined in \eqref{defcbb}. Then exponential stability holds (see \eqref{eq:normalised-stab}).\\
		\end{enumerate}
	\end{corollary}
	
	\begin{proof}
		\Cref{A:BSD} holds with $\Acal_P=\Acal$ since $\Acal \in \Lcal(\Hcal)$, meaning  $Pe^{t\Acal} = e^{t\Acal}P$ for any $t\geq 0$.\\ 
		To show (a) and (b), we apply \cref{thm:dw-bounded}. More precisely, identities \eqref{eq:tau-i_finite} and \eqref{eq:tau-ii_finite} are respectively \eqref{eq:tau-i} and \eqref{eq:tau-ii} with $\Acal_P=\Acal$, $PD=BK$, $\norm{P}=1$. 
		In order to prove (c), we apply \cref{cor:dw-normalised,prop:gamma0-nec}. Then \eqref{eq:tau-iii_finite} is \eqref{eq:tau-iii} with $\Acal_P=\Acal$, $PD=D=BK$, $\norm{P}=1$.
	\end{proof}
	
	\paragraph{Dwell-times for systems with a finite-dimensional unstable part}
	
	\begin{assumption}[Finite-dimensional unstable part]\label{A:split}
		There exist closed subspaces $\Hu,\Hs\subset\Hcal$ with
		\begin{equation*}
			\Hcal=\Hu\oplus\Hs, \qquad \dim\Hu=N<\infty,
		\end{equation*}
		such that the associated projections $\Pu\in\Lcal(\Hcal)$ (onto $\Hu$ along $\Hs$) and $\Ps=I-\Pu$ commute with $e^{t\Acal}$ for every $t\ge0$.\\ 
		The semigroup $e^{t\As}\coloneqq e^{t\Acal}|_{\Hs}$ is exponentially stable, i.e.
		\begin{equation*}
			\norm{e^{t\As}}_{\Lcal(\Hs)}\le M_se^{-\alpha_st} \quad \text{for some } M_s\ge1, \quad \alpha_s>0.
		\end{equation*}
	\end{assumption}
	
	Under \cref{A:split}, $e^{t\Acal}|_{\Hu}$ is a $C_0$-semigroup on the finite-dimensional space $\Hu$, hence of the form $e^{t\Au}$ with $\Au\in\Lcal(\Hu)$. Moreover $\Hu\subset\Dcal(\Acal)$ and $\Acal|_{\Hu}=\Au$.\\ 
	The typical example is the Riesz projector defined by
	\begin{equation*}
		\Pu=\frac1{2\pi i}\oint_\Gamma(\lambda-\Acal)^{-1}\,d\lambda,
	\end{equation*}
	associated with finitely many isolated eigenvalues of finite algebraic multiplicity enclosed by a contour $\Gamma$. No further spectral structure is needed.\\
	We set $B_u\coloneqq\Pu B\in\Lcal(\Ucal,\Hu)$, $B_s\coloneqq\Ps B$.
	
	It belongs to the folklore (see \cite{BDDM92}) that under \cref{A:split}, one can stabilize the finite-dimensional system $\dot z_u=\Acal_uz_u+B_u u$ with a feedback operator $K_u\in\Lcal(\Hu,\Ucal)$. In the following, we assume that $K$ is given as follows.
	\begin{assumption}[Control of the unstable part]\label{A:Ku}
		$K=K_u\Pu$ with $K_u\in\Lcal(\Hu,\Ucal)$ such that
		\begin{equation*}
			\norm{e^{t(\Au+B_uK_u)}}_{\Lcal(\Hu)}\le M_ue^{-\alpha_ut} \quad \text{for some } M_u\ge1, \quad \alpha_u>0.
		\end{equation*}
	\end{assumption}
	
	\begin{lemma}\label{lem:split-stab}
		Under \cref{A:gen,A:BK,A:split,A:Ku}, \cref{A:Kstab} holds.
	\end{lemma}
	\begin{proof}
		Let $z_0\in\Hcal$ and $z(t)=e^{t({\Acal+BK})}z_0$.\\ 
		By \cref{lem:change-gen} (applied to ${\Acal+BK}$ and the perturbation $-BK$), it holds
		\begin{equation*}
			z(t)=e^{t\Acal}z_0+\int_0^te^{(t-s)\Acal}BK_u\Pu z(s)\,ds.
		\end{equation*}
		Applying $\Pu$ and $\Ps$, which commute with $e^{t\Acal}$, and writing $z_u=\Pu z$, $z_s=\Ps z$, we obtain
		\[
		z_u(t)=e^{t\Au}z_u(0)+\int_0^te^{(t-s)\Au}B_uK_uz_u(s)\,ds,\quad z_s(t)=e^{t\As}P_sz_0+\int_0^te^{(t-s)\As}B_sK_uz_u(s)\,ds .
		\]
		The first identity says that $z_u$ is the mild, hence classical, solution of $\dot z_u=(\Au+B_uK_u)z_u$ in the finite-dimensional space $\Hu$ (\cref{lem:change-gen}), so that, from \eqref{A:Ku},
		\begin{equation*}
			\normH{z_u(t)}\le M_ue^{-\alpha_ut}\norm{\Pu}\normH{z_0}.
		\end{equation*} 
		From \cref{A:split}, the second one then gives
		\begin{equation*}
			\normH{z_s(t)}\le M_se^{-\alpha_st}\norm{\Ps}\normH{z_0}+M_sM_u\norm{B_sK_u}\norm{\Pu}\normH{z_0}\int_0^te^{-\alpha_s(t-s)}e^{-\alpha_us}\,ds.
		\end{equation*}
		With $\alpha\coloneqq\min\{\alpha_u,\alpha_s\}$, the last integral is at most $te^{-\alpha t}\le(e\varepsilon)^{-1}e^{-(\alpha-\varepsilon)t}$ for every $\varepsilon\in(0,\alpha)$.\\ 
		Hence $\norm{e^{t({\Acal+BK})}}\le M_\varepsilon e^{-(\alpha-\varepsilon)t}$ for a constant $M_\varepsilon$.
	\end{proof}
	
	\begin{corollary}[Finite-dimensional unstable part]\label{cor:dw-spectral}
		Let \cref{A:gen,A:BK,A:CL,A:split,A:Ku} hold, and monitor the unstable component only, i.e.
		\begin{equation*}
			\HP=\Hu, \quad P=\Pu, \quad D=BK_u\in\Lcal(\Hu,\Hcal),
		\end{equation*}
		so that
		\begin{equation*}
			BK=DP, \quad PD=B_uK_u.
		\end{equation*} 
		Then \cref{A:BSD} holds with $\Acal_P=\Au$, and the following holds.
		\begin{enumerate}[label=\textup{(\alph*)}]
			\item If $\gamma>0$ and $m$ satisfies \eqref{H1}, \eqref{H2}, \eqref{H4}, then a dwell-time is given by 
			\begin{equation} \label{eq:tau-i_finite22}
				\tau_0=\min\left\{1,\frac{e^{-\norm{\Au}}}{(\norm{\Au}+\norm{B_uK_u})(1+\gamma^{-1})+\norm{\Pu}C_e}\right\}.
			\end{equation}
			\item If $\gamma=0$, \cref{A:Lstab} holds, and $m$ satisfies \eqref{H5} with $\mu\le\alpha_L$, $\mu<\alpha_K$ then a dwell-time along each trajectory is given by 
			\[
			\tau_0=\min\left\{1,\ \frac{c_m\,e^{-\mu-\norm{\Au}}}{\norm{\Pu}\big[(\norm{\Au}+\norm{B_uK_u})\,C^*_\mu[\zhat]+\norm{LC}M_L\normH{e_0}\big]}\right\},
			\]
			with $C^*_\mu[\zhat]=\sup_{t\ge0}\{e^{\mu t}\normH{\zhat(t)}\}<\infty$.\\
			If moreover $m$ satisfies \eqref{H3}, the closed loop is exponentially stable at every rate $\lambda$ with $0<\lambda\le\min\{\alpha_L,\mu\}$ and $\lambda<\alpha_K$, without any small-gain condition.
			\item If $\gamma=0$, \cref{A:Lstab} holds, and $m$ is the normalized internal variable \eqref{Def-m-E4} with $\mu\le\alpha_L$, $\mu<\alpha_K$, and $\hat E_0\geq \normH{\zhat_0}+\normH{e_0}>0$.
			Then a dwell-time (independent of the initial datum) is given by 
			$$
			\tau_\star=\min\left\{1,\ \frac{m_0\,e^{-\mu-\norm{\Au}}}{\norm{\Pu}\big[(\norm{\Au}+\norm{B_uK_u})\,\bar C_\mu+\norm{LC}M_L\big]}\right\},
			$$
			with $\bar C_\mu=M_K\left(1+\frac{M_L\norm{LC}}{\alpha_K-\mu}+\frac{\norm{BK_u}m_0}{\alpha_K-\mu}\right)$, 
			and exponential stability holds (see \eqref{eq:normalised-stab}).
		\end{enumerate}
		No regularity of the initial datum is needed.
	\end{corollary}
	
	\begin{proof}
		Since $\Hu$ is invariant we have $\Pu e^{t\Acal}=e^{t\Acal}\Pu=e^{t\Au}\Pu$ so \cref{A:BSD} holds with $\Acal_P=\Au$, $\norm{\Acal_P}=\norm{\Au}$, $PD=\Pu BK_u=B_uK_u$, $\norm P=\norm{\Pu}$ and $\norm D=\norm{BK_u}$.\\ 
		\cref{A:Kstab} holds by \cref{lem:split-stab}.\\
		For (a) and (b) we apply \cref{thm:dw-bounded} - (i) and \cref{thm:stab}.\\
		For (c), \cref{thm:dw-bounded} - (ii) gives the dwell-time, \cref{Lem-Stab} with $\gamma=0$ gives $C^*_\mu[\zhat]<\infty$, and \cref{thm:stab} with $\gamma=0$, for which the small-gain condition is void and \eqref{eq:smallgain} reduces to $\lambda\le\min\{\alpha_L,\mu\}$, $\lambda<\alpha_K$, gives the exponential stability.\\
		For (d) we apply \cref{cor:dw-normalised} and, for the last statement, \cref{prop:gamma0-nec} noticing that $\Au+B_uK_u$ is invertible since it is Hurwitz.
	\end{proof}
	
	\begin{remark}
		In \cref{cor:dw-spectral}, the stable component $\zhat_s=\Ps\zhat$ is not monitored. It is driven by the sampled unstable state through $B_sK_u\zhat_u(t_k)$ and by the output residual $LCe$, and its decay is part of \cref{thm:stab}.
	\end{remark}
	
	\subsection{An obstruction independent of the triggering rule}\label{sec:necessity}
	
	The dwell-times of \cref{sec:dwell-bounded} rely on \cref{A:BSD}. 
	We now show that, for unbounded generators, these restrictions cannot be removed in general, whatever the triggering rule. More precisely, for some systems satisfying all standing assumptions, no triggering rule can achieve both a dwell-time independent of the initial datum and the boundedness of all trajectories. To make the statement as strong as possible, we allow arbitrary rules.
	
	\begin{definition}[Triggering scheme]\label{def:scheme}
		A triggering scheme associates with every initial datum $(\zhat_0,e_0)\in\Hcal\times\Hcal$ a strictly increasing event sequence $0=t_0<t_1<\cdots$ (finite or infinite) and the corresponding function $\zhat$ satisfying \eqref{S1} of \cref{def:solution} on every inter-event interval.\\ 
		No assumption is made on how the events are decided. They may depend on the whole past, on auxiliary variables, on the index $k$, or on anything else.\\ 
		A triggering scheme admits a dwell-time independent of the initial datum if there exists $\tau^*>0$ with $t_{k+1}-t_k\ge\tau^*$ for every initial datum and every admissible $k$ with $t_{k+1}<\infty$.\\ 
		It has bounded trajectories if $\sup_{t\ge0}\normH{\zhat(t)}<\infty$ for every initial datum, and is stabilizing if $\normH{\zhat(t)}\to0$ as $t\to\infty$ for every initial datum.
	\end{definition}
	
	\begin{example}[An unstable diagonal system]\label{ex:counter}
		Let $\Hcal=\Ucal=\Ycal=\ell^2(\N^*;\Cbb)$ with canonical basis $(\phi_n)_{n\ge1}$, $\Acal\phi_n=(1+in)\phi_n$ with
		\begin{equation*}
			\Dcal(\Acal)=\{z \in \Hcal,\quad \sum n^2|z_n|^2<\infty\}, \qquad B=C=I, \qquad K=L=-2I.
		\end{equation*} 
		Then \cref{A:gen,A:BK,A:Kstab,A:CL,A:Lstab} hold with sharp constants. More precisely, $\norm{e^{t\Acal}}=e^t$, and ${\Acal+BK}={\Acal+LC}=\Acal-2I$ is diagonal with eigenvalues $-1+in$, so that
		\begin{equation*}
			\norm{e^{t({\Acal+BK})}}=\norm{e^{t({\Acal+LC})}}=e^{-t} \quad \text{i.e. }M_K=M_L=1,\quad  \alpha_K=\alpha_L=1.
		\end{equation*}
		Every mode is unstable for the free dynamics.
	\end{example}
	
	\begin{lemma}[Reduction to a scalar equation]\label{lem:reduction}
		In \cref{ex:counter}, let $n\ge1$, $\zhat_0=z_0=\phi_n$, so that $e_0=0$. For any triggering scheme, $\zhat(t)=a(t)\phi_n$ for $t\ge0$, where $a(0)=1$ and, for every admissible $k$ and $h\ge0$ with $t_k+h\le t_{k+1}$, it holds
		\begin{equation}\label{eq:scalar}
			a(t_k+h)=\Big[\frac{-1+in}{1+in}e^{(1+in)h}+\frac2{1+in}\Big]a(t_k),\qquad
			|a(t_k+h)|\ge\Big(e^h-\frac2{\sqrt{1+n^2}}\Big)|a(t_k)| .
		\end{equation}
	\end{lemma}
	\begin{proof}
		Since $e_0=0$, $e\equiv0$ from \cref{lem:error}. The operators $B,K,C,L$ are multiples of the identity, so they leave $\spanop\{\phi_n\}$ invariant, as does $e^{t\Acal}$. 
		By \eqref{S1} of \cref{def:solution} and induction on $k$, $\zhat(t)\in\spanop\{\phi_n\}$ for all $t$, whatever the event sequence. Writing $\zhat=a\phi_n$, \eqref{S1} reads
		\begin{equation*}
			a(t_k+h)=e^{(1+in)h}a(t_k)-2a(t_k)\int_0^he^{(1+in)(h-s)}\,ds=a(t_k)\left[e^{(1+in)h}-\frac{2(e^{(1+in)h}-1)}{1+in}\right],
		\end{equation*}
		which is the identity in \eqref{eq:scalar} since $1-2\left(1+in\right)^{-1}=\left(-1+in\right)\left(1+in\right)^{-1}$.\\ 
		The modulus of the first term in the bracket is $e^h$, that of the second is $2(\sqrt{1+n^2})^{-1}$, and the reverse triangle inequality gives the bound in estimate in \eqref{eq:scalar}.
	\end{proof}
	
	\begin{theorem}[Obstruction]\label{thm:necessity}
		For the system of \cref{ex:counter}, no triggering scheme admits a dwell-time independent of the initial datum and has bounded trajectories.\\ 
		In particular, no stabilizing triggering scheme admits a dwell-time independent of the initial datum.
	\end{theorem}
	\begin{proof}
		Assume that a triggering scheme admits a dwell-time $\tau^*>0$ independent of the data.\\ Choose $n$ such that
		\begin{equation*}
			q_n\coloneqq e^{\tau^*}-\frac{2}{\sqrt{1+n^2}}>1,
		\end{equation*}
		and apply the scheme to $\zhat_0=z_0=\phi_n$.\\ 
		We claim that $a(t_k)\ne0$ and $|a(t_{k+1})|\ge q_n|a(t_k)|$ for every admissible $k$ with $t_{k+1}<\infty$.\\ 
		Indeed $a(t_0)=a(0)=1$, and if $a(t_k)\ne0$ and $t_{k+1}<\infty$, then $t_{k+1}-t_k\ge\tau^*$ and, since $h\mapsto e^h-2(\sqrt{1+n^2})^{-1}$ is increasing, \eqref{eq:scalar} gives
		\begin{equation*}
			|a(t_{k+1})|\ge q_n|a(t_k)|>0.
		\end{equation*}
		If the event sequence is infinite, then $|a(t_k)|\ge q_n^k \xrightarrow[k \to \infty]{}\infty$.\\ 
		If it is finite with last event $t_K$, then \eqref{eq:scalar} holds for every $h\ge0$ at $k=K$, with $a(t_K)\ne0$, and
		\begin{equation*}
			|a(t_K+h)|\ge\left(e^h-\frac{2}{\sqrt{1+n^2}}\right)|a(t_K)| \xrightarrow[h \to \infty]{}\infty
		\end{equation*} 
		In both cases, since $\zhat(t)=a(t)\phi_n$, $\sup_t\normH{\zhat(t)}=\infty$.
	\end{proof}
	
	\begin{proof}[Proof of \cref{thm:main-dw}]
		The system of \cref{ex:counter} has a generator which is not bounded and satisfies \cref{A:gen,A:BK,A:Kstab,A:CL,A:Lstab}, and \cref{thm:necessity} applies to it.
	\end{proof}
	
	\begin{remark}[Without observer]\label{rem:no-observer}
		Assume that $e_0=0$, i.e. the observer is initialized at the true state, so that $\zhat=z$, or equivalently that the observer is not needed. This is the setting of \cite{BaudouinErvedoza2025}, with $K=K'C$ and $P=C$ (see \cref{sec:intro}). Then $LCe\equiv0$, so that \eqref{H4} holds with $C_e=0$ for every internal variable, and \cref{A:Lstab} plays no role. In particular, the implementability issues discussed in \cref{rem:gamma0} (and later in \cref{rem:implement}) disappear.
		\begin{enumerate}[label=\textup{(\alph*)},leftmargin=*]
			\item If $\gamma>0$ and \cref{A:BSD} holds, \cref{thm:dw-bounded} - (i) gives the dwell-time
			\[
			\tau_0=\min\Big\{1,\ \frac{e^{-\norm{\Acal_P}}}{(\norm{\Acal_P}+\norm{PD})(1+\gamma^{-1})}\Big\},
			\]
			which is independent of the initial datum and of the internal variable. In particular, the simple choice $m(t)=m_0e^{-\mu t}$ (see (E1) below) gives a dwell-time independent of the data together with exponential stability (\cref{thm:stab}).
			\item If $\gamma=0$ and $z_0$ is known, as is implicitly the case for rules involving $\normH{z(t)}$ such as the one of \cite{BaudouinErvedoza2025}, the normalized internal variable $m(t)=m_0e^{-\mu t}\normH{z_0}$ is implementable, and \cref{cor:dw-normalised} gives a dwell-time independent of the data.
		\end{enumerate}
		On the other hand, \cref{A:BSD} cannot be removed in general. The obstruction of \cref{thm:necessity} is obtained for data with $e_0=0$, so that it also holds without observer. For an unbounded generator, a dwell-time independent of the  initial datum may be incompatible with bounded trajectories even when the full state is available.\end{remark}
	
	\section{Admissible internal variables}\label{sec:thresholds}
	
	The following result presents different possibilities of the internal variable. 
	
	\begin{proposition}[Standard internal variables]\label{prop:examples}
		Let \cref{A:gen,A:CL,A:Lstab} hold and let $e(t)=e^{t({\Acal+LC})}e_0$ for every $t\geq 0$.
		\begin{enumerate}[label=\textup{(E\arabic*)},leftmargin=*]
			\item \textup{(Exponential)} The function
			\begin{equation*}
				m(t)=m_0e^{-\mu t}, \quad t\geq 0, \quad \text{with } m_0,\mu>0,
			\end{equation*}
			satisfies \eqref{H1}--\eqref{H3} with $C_m=m_0$, \eqref{H5} with $c_m=m_0$, and, 
			if $\mu\le\alpha_L$, \eqref{H4} with the constant $C_e=e^{\mu}\norm{LC}M_L\normH{e_0}m_0^{-1}$.
			
			\item \textup{(Output residual filter)} Let $m_0>0$, $\beta\ge0$, $0<\mu<\alpha_L$, and let $m$ follow the dynamic
			\begin{equation*}
				\dot m(t)=-\mu m(t)+\beta\norm{Ce(t)}_{\Ycal}, \quad t\geq 0, \quad \text{with } m(0)=m_0,
			\end{equation*}
			that is, $m(t)=m_0e^{-\mu t}+\beta\displaystyle\int_0^te^{-\mu(t-s)}\norm{Ce(s)}_{\Ycal}\,ds$. Then
			\[
			m_0e^{-\mu t}\le m(t)\le\Big(m_0+\frac{\beta\norm{C}M_L\normH{e_0}}{\alpha_L-\mu}\Big)e^{-\mu t},\qquad t\ge0,
			\]
			so that $m$ satisfies \eqref{H1}--\eqref{H3} with $C_m=m_0+\beta\norm{C}M_L\normH{e_0}(\alpha_L-\mu)^{-1}$, \eqref{H5} with the constant  $c_m=m_0$, and \eqref{H4} with $C_e=e^{\mu}\norm{LC}M_L\normH{e_0}m_0^{-1}$.
			
			\item \textup{(Supremum output residual filter)} Let $m_0,\mu,\eta,\nu>0$ and
			\begin{equation}\label{eq:supfilter}
				m(t)=m_0e^{-\mu t}+\eta 
				\sup_{0\le s\le t}\{ e^{-\nu(t-s)}\normH{LCe(s)}\}, \quad t\geq 0.
			\end{equation}
			Then $m$ satisfies \eqref{H1}--\eqref{H3} with $C_m=m_0+\eta\norm{LC}M_L\normH{e_0}$ and rate $\min\{\mu,\nu,\alpha_L\}$, \eqref{H5} with $c_m=m_0$, and \eqref{H4} with $C_e=e^{\nu}\eta^{-1}$, independently of the initial datum $(\zhat_0, e_0)$.
			
			\item \textup{(Normalized)} The function
			\begin{equation*}
				m(t)=m_0e^{-\mu t}
				\hat E_0,
				\quad t\geq 0, \quad \text{with } m_0,\mu>0,
				\quad \hat E_0 \geq \normH{\hat z_0}+\normH{e_0},
			\end{equation*}
			satisfies \eqref{H1}--\eqref{H3} with $C_m=m_0(\normH{\zhat_0}+E_0)$, \eqref{H5} with $c_m=m_0 \hat E_0$.
			If $\mu\le\alpha_L$, it satisfies \eqref{H4} with $C_e=e^{\mu}\norm{LC}M_Lm_0^{-1}$, independently of the initial datum $(\zhat_0, e_0)$.
		\end{enumerate}
	\end{proposition}
	
	\begin{proof}
		(E1) - The proof of claims \eqref{H1}-\eqref{H2}-\eqref{H3}-\eqref{H5} is immediate. We only show here that \eqref{H4} is satisfied.
		If $s\in[\max(t-1,0),t]$, then $s\ge t-1$, and by \cref{lem:error} and $\mu\le\alpha_L$, we deduce
		\begin{equation*}
			\normH{LCe(s)}\le\norm{LC}M_L\normH{e_0}e^{-\mu s}\le\norm{LC}M_L\normH{e_0}e^{\mu}e^{-\mu t}.
		\end{equation*}
		
		(E2) - The variation of constants formula gives the expression of $m$, which is continuous since $s\mapsto\norm{Ce(s)}_{\Ycal}$ is. The lower bound is clear since $\beta\ge0$.\\ 
		For the upper bound, $\norm{Ce(s)}_{\Ycal}\le\norm{C}M_L\normH{e_0}e^{-\alpha_Ls}$, so that
		\[
		\int_0^te^{-\mu(t-s)}\norm{Ce(s)}_{\Ycal}\,ds\le\norm{C}M_L\normH{e_0}e^{-\mu t}\int_0^te^{-(\alpha_L-\mu)s}\,ds\le\frac{\norm{C}M_L\normH{e_0}}{\alpha_L-\mu}e^{-\mu t}.
		\]
		\eqref{H4} follows from the lower bound as in (E1).
		
		(E3) - The function $t\mapsto \sup_{0\le s\le t}\{ e^{-\nu(t-s)}\normH{LCe(s)}\}$ is continuous as $s \mapsto \normH{L Ce(s)}$ is continuous.
		Hence $m$ is continuous, and it is positive since $m_0>0$.\\ 
		By \cref{lem:error}, we obtain
		\begin{equation*}
			e^{\nu s}\normH{LCe(s)}\le\norm{LC}M_L\normH{e_0}e^{(\nu-\alpha_L)s}.
		\end{equation*}
		On the one hand, if $\nu\ge\alpha_L$, we have 
		\begin{equation*}
			\sup_{0\le s\le t}\{ e^{-\nu(t-s)}\normH{LCe(s)}\} \leq e^{-\nu t}\norm{LC}M_L\normH{e_0}e^{(\nu-\alpha_L)s}\le\norm{LC}M_L\normH{e_0}e^{-\alpha_Lt},
		\end{equation*}
		and one the other hand, if $\nu<\alpha_L$,
		\begin{equation*}
			\sup_{0\le s\le t}\{ e^{-\nu(t-s)}\normH{LCe(s)}\}\leq e^{-\nu t}\norm{LC}M_L\normH{e_0}e^{(\nu-\alpha_L)t}\le\norm{LC}M_L\normH{e_0}e^{-\nu t}.
		\end{equation*}
		This gives \eqref{H3}. For \eqref{H4}, if $s\in[\max(t-1,0),t]$, then
		\begin{equation*}
			m(t)\ge\eta e^{-\nu(t-s)}\normH{LCe(s)}\ge\eta e^{-\nu}\normH{LCe(s)}.
		\end{equation*}
		
		(E4) - The proof of claims \eqref{H1}-\eqref{H2}-\eqref{H3}-\eqref{H5} is immediate, using \cref{lem:conv-exp} when needed. For \eqref{H4}, as in (E1), since $\mu\le\alpha_L$ and $\normH{e_0}\le \hat E_0$, we have, for $s\in[\max(t-1,0),t]$,
		\begin{equation*}
			\normH{LCe(s)}\le\norm{LC}M_L\normH{e_0}e^{\mu}e^{-\mu t}\le e^{\mu}\norm{LC}M_Lm_0^{-1}\,m(t).
		\end{equation*}
	\end{proof}
	
	\begin{remark}[Implementability]\label{rem:implement}
		The variables (E1)--(E3) are implementable. (E2) and (E3) only use the output residual $Ce=C\zhat-y$ (or $LCe$), which is computed from the observer and the available measurement.\\ 
		The variable (E4) requires an estimate on $\normH{e_0}$. It becomes implementable if an a priori bound $E_0\ge\normH{e_0}$ is known.
		For $\gamma = 0$, this is the only example we give of an internal variable for which we obtain a dwell-time which is independent of the initial datum, when \cref{A:BSD} is verified, see \cref{cor:dw-normalised}.
	\end{remark}

	\begin{remark}
		When $e_0=0$, $m$ in (E1), (E2) and (E3)  are independent of the initial state of the observer $\zhat_0$.  Thus, \cref{prop:gamma0-nec} applies. Thus, when $\gamma=0$ and  $m$ is defined  as in (E1), (E2) and (E3),  there is no dwell-time independent of the data.
	\end{remark}
	
	\begin{remark}
		\begin{itemize}
			\item In the context of \cref{cor:dw-finite} - (a), for the supremum output residual filter (E3), we have $C_e=e^\nu\eta^{-1}$. Then the dwell-time $\tau_0$ in \eqref{eq:tau-i_finite} is independent of the initial datum. If moreover \cref{A:Kstab,A:Lstab} hold and
			\begin{equation*}
				M_K\norm{BK}\gamma<\alpha_K-\lambda \quad \text{with }0<\lambda\le\min\{\alpha_L,\mu,\nu\}, \quad \lambda<\alpha_K,
			\end{equation*}
			the closed loop is also exponentially stable at rate $\lambda$ (\cref{thm:stab}).		
			\item In the context of \cref{cor:dw-spectral}  - (a), for the supremum output residual filter (E3), we have $C_e=e^\nu\eta^{-1}$. Then, the dwell-time $\tau_0$ in \eqref{eq:tau-i_finite22} is independent of the initial datum. Moreover, if \cref{A:Lstab} holds and
			\begin{equation*}
				M_K\norm{BK_u}\norm{\Pu}\gamma<\alpha_K-\lambda \quad \text{with } 0<\lambda\le\min\{\alpha_L,\mu,\nu\}, \quad \lambda<\alpha_K,
			\end{equation*}
			then the closed loop is exponentially stable at rate $\lambda$ (\cref{thm:stab}).	
		\end{itemize}
	\end{remark}
	
	\begin{table}[h!]
		\centering
		\label{tab:proposition}
		\begin{tabular}{|l| c| c | c | c| c|}
			\toprule
			\textbf{$m(t)$} & $C_m$ & $c_m$  & $C_e$ & Unif. dwell-time\\
			\midrule
			
			\textbf{(E1)}
			&
			$m_0$
			&
			$m_0$
			&
			$e^{\mu}\norm{LC}M_L\normH{e_0}m_0^{-1}$
			&
			No
			\\
			\midrule
			\textbf{(E2)}
			&
			$m_0 + \dfrac{\beta\norm{C}M_L\normH{e_0}}{\alpha_L - \mu}$
			&
			$m_0$
			&
			$e^{\mu}\norm{LC}M_L\normH{e_0}m_0^{-1}$
			&
			No
			\\
			\midrule
			\textbf{(E3)}
			&
			$m_0 + \eta\norm{LC}M_L\normH{e_0}$
			&
			$m_0$
			&
			$e^{\nu}\eta^{-1}$
			&
			Yes
			\\
			\midrule
			\textbf{(E4)}
			&
			$m_0 \hat E_0$
			&
			$m_0 \hat E_0$
			&
			$e^{\mu}\norm{LC}M_L m_0^{-1}$
			&
			Yes
			\\
			\bottomrule
		\end{tabular}
		\caption{Summary of the values of $C_m$, $c_m$ and $C_e$ for the standard internal variables from \cref{prop:examples}.
			The last column indicates if there is a uniform dwell time or not, in which case the dwell time provided by \cref{thm:dw-bounded} gives a dwell-time which depends on $\normH{e_0}$.}
	\end{table}
	
	\section{Applications and numerical illustrations}\label{sec:examples}\label{sec:numerics}
	
	We apply the previous results to four classes of systems, from finite dimension to PDE with unbounded generators, and we illustrate each case numerically.\\
	In each situation, we compare the event-triggered implementation of the observer-based feedback with its continuous-time version $u=K\zhat(t)$. We use the threshold \eqref{eq:threshold} with the internal variables (E1) and (E3) of \cref{prop:examples}, as well as the threshold reduced to the internal variable ($\gamma=0$). Unless stated otherwise, the parameters of the internal variables are
	\[
	m_0=0.05,\qquad \mu=0.5,\qquad \eta=0.1,\qquad \nu=1,
	\]
	and the observer is started at $\zhat_0=0$, so that $e_0=-z_0$. 
	
	In the figures below, $\gamma_{\mathrm{th}}$ denotes a value of $\gamma$ satisfying the small-gain condition \eqref{eq:smallgain}, computed from the constants of \cref{A:Kstab,A:Lstab} for the discretized system, and $t_1$ denotes the first event time. 
	
	\subsection{A finite-dimensional case}\label{sec:num-fd}
	Let $\Hcal=\R^3$, $\Ucal=\Ycal=\R$ and
	\[
	\Acal=\begin{pmatrix}0&1&0\\0&0&1\\2&1&-2\end{pmatrix},\qquad B=\begin{pmatrix}1&0&0\end{pmatrix}^\top,\qquad C=\begin{pmatrix}0&0&1\end{pmatrix}.
	\]
	The characteristic polynomial of $\Acal$ is $\lambda^3+2\lambda^2-\lambda-2=(\lambda-1)(\lambda+1)(\lambda+2)$, so the free dynamics has the unstable eigenvalue $1$. 
	The controllability matrix $(B,\Acal B,\Acal^2B)$ has columns $(1,0,0)^\top$, $(0,0,2)^\top$, $(0,2,-4)^\top$ and determinant $-4$.
	The observability matrix $(C, C\Acal, C\Acal^2)$ has rows $(0,0,1)$, $(2,1,-2)$, $(-4,0,5)$ and determinant $4$.
	Hence $K$ and $L$ can be chosen (for instance by pole placement) so that \cref{A:Kstab,A:Lstab} hold.\\ 
	We place the spectra of $\Acal+BK$ and $\Acal+LC$ at $\{-1,-2,-3\}$ and $\{-3,-4,-5\}$, which gives	$\gamma_{\mathrm{th}}=0.015$, and we take $z_0=(1,0.5,-1)$.\\ 
	
	\begin{figure}[ht]
		\includegraphics[width=\textwidth]{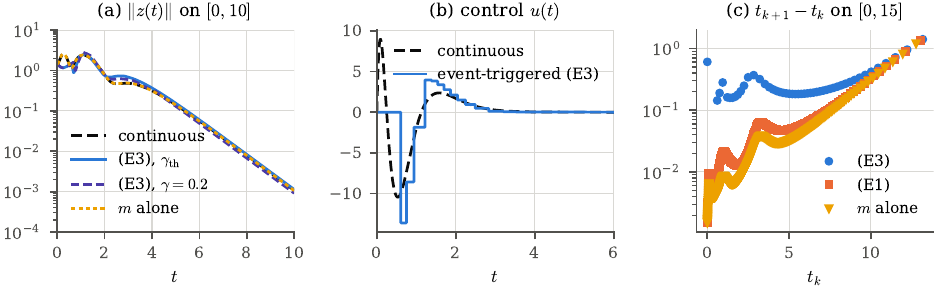}
		\caption{Finite-dimensional system, $z_0=(1,0.5,-1)$ and $\zhat_0=0$. (a) $\normH{z(t)}$ on $[0,10]$ for the continuous feedback $u=K\zhat$ of \eqref{eq:continuous} and several thresholds; (b) control $u(t)$ continuous and event-triggered with (E3); (c) inter-event times $t_{k+1}-t_k$ against $t_k$ on $[0,15]$.}\label{fig:num-fd}
	\end{figure}
	
	On \cref{fig:num-fd}(a), all event-triggered loops recover the decay rate of the continuous feedback, as predicted by \Cref{thm:stab}. This also holds for $\gamma=0.2$, which violates \eqref{eq:smallgain} so that the small-gain condition is sufficient but not necessary. 
	
	On \cref{fig:num-fd}(c), the inter-event times differ at the beginning. When $\zhat_0=0$, the supremum output residual filter (E3) makes the threshold follow $LCe$, thereby avoiding the very short inter-event times observed with (E1) and $m$ alone.
	
	\begin{figure}[ht]
		\includegraphics[width=\textwidth]{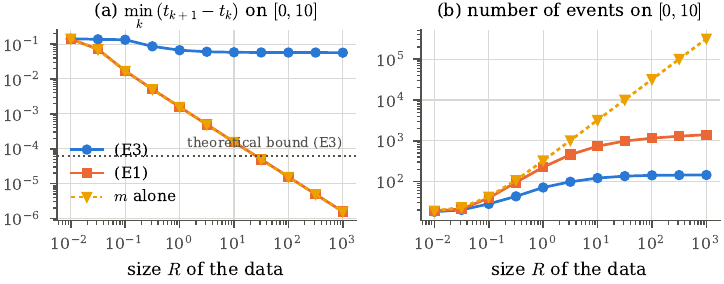}
		\caption{Finite-dimensional system. (a) Minimal inter-event time and (b) number of events on $[0,10]$ for initial data of size $R$ ($z_0=Rh$, $\zhat_0=0$); the dotted line in (a) is the explicit dwell-time \eqref{eq:tau-i_finite} for (E3).}\label{fig:num-fd-dwell}
	\end{figure}
	
	\cref{fig:num-fd-dwell}(a) shows the minimal inter-event time as a function of the size $R$ of the initial data. With (E3), the minimal inter-event time does not depend on the size $R$ of the data, and the number of events bounded (\cref{fig:num-fd-dwell}(b)). This is the dwell-time independent of the data of \cref{thm:dw-bounded}~(i) and \cref{cor:dw-finite}.
	With (E1) and with $m$ alone, the minimal inter-event time decreases like $R^{-1}$, so that the dwell-time depends on the data. The output residual grows with $R$, while $m_0$ is fixed.
	
	\subsection{A parabolic model}\label{sec:num-parab}
	
	Consider the following reaction-diffusion equation 
	\begin{equation*}
		\partial_tz=\kappa\Delta z+p  z+\chi_\omega u,
	\end{equation*} 
	on a smooth bounded domain $\Omega$ with Dirichlet boundary conditions, $\kappa>0$, $p \ge0$, with output $y=\chi_{\omega_o}z$, where $\omega,\omega_o\subset\Omega$ are nonempty open sets.\\ 
	Let $\Hcal=\Ucal=\Ycal=L^2(\Omega)$ and define
	\[
	\Acal=\kappa\Delta+p  I,\qquad
	\Dcal(\Acal)=H^2(\Omega)\cap H_0^1(\Omega),\qquad
	B=\chi_\omega,\qquad C=\chi_{\omega_o},
	\]
	so that the system reads $\dot z=\Acal z+Bu$, $y=Cz$.\\
	The generator is self-adjoint with compact resolvent, so that only finitely many eigenvalues lie above any given level, and \cref{A:split} holds with the orthogonal projection $\Pu$ onto the eigenspaces associated with eigenvalues larger than $\lambda_c$, for any $\lambda_c>0$.\\ 
	Since eigenfunctions of the Dirichlet Laplacian cannot vanish on a nonempty open set (unique continuation, see for instance \cite{LeRousseauLebeau2012}), the restrictions to $\omega$ (respectively $\omega_o$) of a basis of each eigenspace are linearly independent, and the Hautus test shows that $(\Au,\Pu\chi_\omega)$ is controllable and $(\Au,\chi_{\omega_o}|_{\Hu})$ is observable. Hence \cref{A:Ku} holds for a suitable $K_u$.\\ 
	Let $L=\iota_uL_u$, with $\iota_u$ the injection of $\Hu$, and $L_u$ such that $\Au+L_u\chi_{\omega_o}|_{\Hu}$ is Hurwitz. Because $\Ps L=0$, the matrix ${\Acal+LC}$ is block-triangular in $\Hu\oplus\Hs$ and its diagonal blocks are exponentially stable. Hence \cref{A:Lstab} holds by the argument of \cref{lem:split-stab}. \Cref{cor:dw-spectral} then provides a dwell-time independent of the initial datum together with exponential stability, by monitoring only the finitely many unstable modes of the observer.
	
	We illustrate this on $\Hcal=L^2(0,\pi)$ with
	\[
	\partial_tz=\partial_{xx}z+5z+\chi_{(0.3,1.2)}u,\qquad y=\int_{1.8}^{2.6}z\,dx,\qquad z(t,0)=z(t,\pi)=0,
	\]
	where $u$ and $y$ are scalar. The PDE is discretized in space only, by a modal Galerkin method  with $40$ Dirichlet modes. The eigenvalues $5-n^2$ give two unstable modes. The gains $K$ and $L$ only act on these two modes, and we compare the monitoring of the unstable part only, $P=\Pu$ (\cref{cor:dw-spectral}), with the monitoring of the whole state, $P=I$. The initial state is $x \mapsto z_0(x)=3x(\pi-x)/\pi^{2}+0.5\sin(5x)$.
	
	\begin{figure}[ht]
		\includegraphics[width=\textwidth]{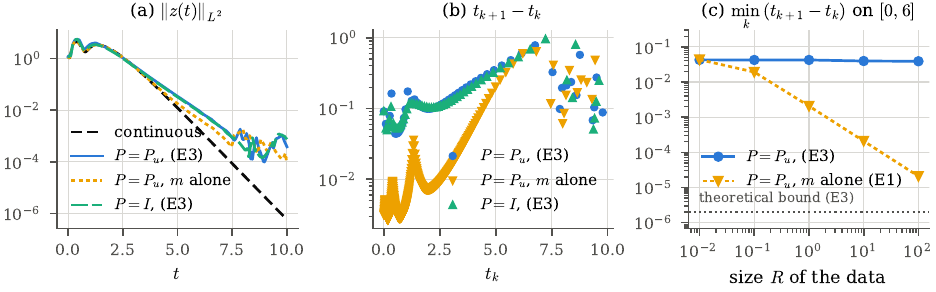}
		\caption{Unstable reaction--diffusion equation, $x \mapsto z_0(x)=3x(\pi-x)/\pi^{2}+0.5\sin(5x)$ and $\zhat_0=0$. (a) $\norm{z(t)}_{L^2}$; (b) inter-event times $t_{k+1}-t_k$ against $t_k$; (c) minimal inter-event time on $[0,6]$ for initial data of size $R$.}\label{fig:num-rd}
	\end{figure}
	
	All event-triggered loops are exponentially stable (\cref{fig:num-rd}(a)).
	Their decay rate is smaller than that of the continuous feedback, because it is limited by the decay of the internal variable, according to  the condition $\lambda\le\mu$ of \cref{thm:stab}. For this smooth datum, monitoring the unstable part or the whole state gives similar inter-event times, while $m$ alone again produces many events at the beginning (\cref{fig:num-rd}(b)). 
	
	The difference appears for initial data whose size $R$ is large. More precisely, with $P=\Pu$ and (E3), the minimal inter-event time does not depend on the size of the data, whereas it decreases like $R^{-1}$ with $m$ alone (\cref{fig:num-rd}(c)). Monitoring only the unstable modes thus gives a dwell-time independent of the data, although the generator is unbounded (\cref{cor:dw-spectral}). When the whole state is monitored, a datum $\zhat_0=z_0$ (so that $e_0=0$) concentrated on a high mode $\sin(nx)$ moves fast under the heat semigroup, and $t_1$ decreases like $n^{-2}$. With $P=\Pu$, such a datum produces no event at all, since $\Pu\zhat$ stays equal to $0$.
	
	\subsection{Wave equations}\label{sec:num-wave}\label{sec:num-counter}
	This section has two goals. First, we explain why the dwell-time results of \cref{sec:dwell-bounded} do not apply to wave equations stabilized by an in-domain damping. Second, we illustrate the obstruction of \cref{thm:necessity} on the system of \cref{ex:counter}, which models an anti-damped wave equation.
	
	\paragraph{Damped wave equations.} Consider the second-order equation
	\begin{equation*}
		\partial_{tt}z=c^2\Delta z-p \partial_tz+\chi_\omega u,
	\end{equation*}
	on a smooth bounded domain $\Omega$ with Dirichlet boundary conditions, $c\neq0$, $p \ge0$ and $\omega,\omega_o\subset\Omega$ nonempty open sets. Let $\Hcal=H^1_0(\Omega)\times L^2(\Omega)$ endowed with the energy norm
	\[
	\norm{(z,v)}_{\Hcal}^2=\norm{\nabla z}_{L^2(\Omega)}^2+\norm{v}_{L^2(\Omega)}^2,
	\]
	and write the equation as a first-order system for $(z,\partial_tz)$, i.e.
	\[
	\partial_t\begin{pmatrix}z\\ \partial_tz\end{pmatrix}
	=\Acal\begin{pmatrix}z\\ \partial_tz\end{pmatrix}
	+Bu,
	\]
	where
	\(\Acal=\begin{pmatrix}0&I\\ c^2\Delta&-p  I\end{pmatrix}\),
	\(\Dcal(\Acal)=\bigl(H^2(\Omega)\cap H^1_0(\Omega)\bigr)\times H^1_0(\Omega)\),
	\(B=\begin{pmatrix}0 \\ \chi_\omega\end{pmatrix}\),  and \(\Ucal=L^2(\Omega)\).\\
	We measure the velocity on $\omega_o$, that is, $\Ycal=L^2(\Omega)$ and $C=\begin{pmatrix}0&\chi_{\omega_o}\end{pmatrix}$, so that $y=\chi_{\omega_o}\partial_tz$, and we take $K=-kB^*$ and $L=-\ell C^*$ with $k,\ell>0$. The control $u=K\zhat=-k\chi_\omega\partial_t\zhat$ is an in-domain damping, and ${\Acal+BK}$ (respectively ${\Acal+LC}$), is the generator of the wave equation with the damping $p+k\chi_\omega$ respectively $p+\ell\chi_{\omega_o}$). If $p>0$, or if $\omega$ and $\omega_o$ satisfy the Geometric Control Condition (GCC), \cref{A:Kstab,A:Lstab} hold \cite{BardosLebeauRauch1992}. \\
	The generator $\Acal$ is unbounded. When $p=0$, its eigenvalues are $\pm i|c|\sqrt{\lambda_j}$, where $(\lambda_j)_{j\ge1}$ are the eigenvalues of the Dirichlet Laplacian $-\Delta$. They lie on the imaginary axis, with frequencies going to infinity. No mode is damped by the free dynamics, so that the feedback has to act on infinitely many modes, and the whole state is monitored ($P=I$). Then \cref{thm:main-wp,thm:stab} apply, but \cref{A:BSD} fails. Moreover, \cref{lem:nonunif} shows that the first term of the increment identity \eqref{eq:incr-identity} has no uniform modulus of continuity over bounded sets of data. The dwell-time results of \cref{sec:dwell-bounded} thus do not apply.
	
	\paragraph{Link with \cref{ex:counter}.} The system of \cref{ex:counter} models a one-dimensional wave equation in which the damping is replaced by an anti-damping. Take $\Omega=(0,\pi)$, $c=1$, and consider
	\begin{equation*}
		\partial_{tt}w=\partial_{xx}w+2\partial_tw-w,\qquad w(t,0)=w(t,\pi)=0.
	\end{equation*}
	With $w(t,x)=e^{\lambda t}\sin(nx)$, we obtain $(\lambda-1)^2=-n^2$, so that the eigenvalues are $1\pm in$, $n\ge1$. The eigenvectors $\big(\sin(nx),(1+in)\sin(nx)\big)$, $n\ge1$, of the branch $1+in$ are orthogonal in the energy space. Their normalizations $\phi_n$ span a closed invariant subspace, on which the generator acts as $\Acal\phi_n=(1+in)\phi_n$. This is the operator of \cref{ex:counter}, written in the basis $(\phi_n)_{n\ge1}$. This model keeps the two features of the wave equation that matter here, namely that all the modes have the same growth rate, and that their frequencies $n$ are unbounded, so that $e^{t\Acal}$ is not uniformly continuous (\cref{lem:nonunif}). Actuation and observation are global so that $B=C=I$ act on every mode, and $K=L=-2I$ move every eigenvalue from $1+in$ to $-1+in$.
	
	We take $\zhat_0=z_0=\phi_n$, so that $e_0=0$ and the output residual $LCe$ vanishes. By \cref{lem:reduction}, for every triggering rule, $\zhat(t)=a(t)\phi_n$ with $a(0)=1$ and, on every inter-event interval,
	\begin{equation*}
		a(t_k+h)=g_n(h)\,a(t_k),\qquad g_n(h)\coloneqq\frac{-1+in}{1+in}\,e^{(1+in)h}+\frac{2}{1+in},\qquad 0\le h\le t_{k+1}-t_k,
	\end{equation*}
	which is \eqref{eq:scalar}. Thus $|g_n(h)|$ is the factor by which the amplitude of the mode $n$ is multiplied when the control is held during a time $h$.
	
	We compare the event-triggered rule with $\gamma=0.25$ and (E1), for which \cref{thm:stab} applies, with periodic sampling of period $\tau$ ($t_k=k\tau$), which has a dwell-time independent of the data.
	
	\begin{figure}[ht]
		\includegraphics[width=\textwidth]{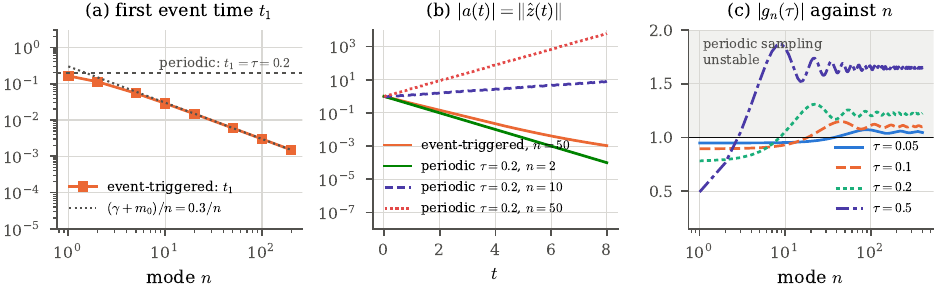}
		\caption{System of \cref{ex:counter}, data $\zhat_0=z_0=\phi_n$ ($e_0=0$). (a) First event time $t_1$ of the event-triggered rule, compared with the period $\tau=0.2$ of the periodic sampling; (b) $\normH{\zhat(t)}=|a(t)|$ for the event-triggered rule ($n=50$) and for periodic sampling with $\tau=0.2$ ($n=2,10,50$); (c) factor $|g_n(\tau)|$ of \eqref{eq:scalar} by which the mode $n$ is multiplied over one sampling interval of length $\tau$, for several $\tau$; periodic sampling of period $\tau$ is unstable for the modes with $|g_n(\tau)|>1$ (shaded region).}\label{fig:num-counter}
	\end{figure}
	
	\Cref{fig:num-counter} displays the dichotomy of \cref{thm:necessity}. The event-triggered loop is stable for every $n$ (see \cref{fig:num-counter}(b) for $n=50$), but $t_1$ decreases like $n^{-1}$ (\cref{fig:num-counter}(a)). Indeed, the mode $n$ oscillates at frequency $n$, so that $\zhat$ moves away from $\zhat_0$ at a speed of order $n$, and reaches the threshold after a time of order $n^{-1}$. Since all the data $\phi_n$ have unit norm, there is no dwell-time independent of the data. Conversely, periodic sampling multiplies the mode $n$ by $g_n(\tau)$ at each sampling instant, and it diverges for the high modes, for which $|g_n(\tau)|>1$ (\cref{fig:num-counter}(b)). \Cref{fig:num-counter}(c) shows that this holds for every period. Holding the control during a fixed time amplifies all sufficiently high modes, which is the argument of the proof of \cref{thm:necessity}.
	
	\subsection{A coupled wave--heat system}\label{sec:num-wh}
	
	We finally consider the transmission problem between a wave equation of state $(w, \partial_tw)$ evolving on $(-1,0)$ and a heat equation of state  $\theta$ evolving on $(0,1)$ \cite{ZhangZuazua2003}
	\begin{equation*}
		\begin{cases}
			\partial_{tt}w(t,x)=\partial_{xx}w(t,x)+\chi_{(-1,0)}u(t,x), & x\in(-1,0), t\geq 0,\\
			\partial_t\theta(t,x)=\partial_{xx}\theta(t,x), & x\in(0,1), t\geq 0,\\
			w(t,-1)=0,\quad\theta(t,1)=0,&  t\geq 0,\\
			\partial_tw(t,0)=\theta(t,0),\quad\partial_xw(t,0)=\partial_x\theta(t,0), &  t\geq 0.
		\end{cases}
	\end{equation*}
	in the energy space $\Hcal=\{w\in H^1(-1,0),\ w(-1)=0\}\times L^2(-1,0)\times L^2(0,1)$.\\
	Its generator is neither skew-adjoint nor analytic, since it couples a conservative wave part with a dissipative heat part.
	Without control, the energy decays only polynomially \cite{ZhangZuazua2004,BattyPaunonenSeifert2016}.\\ 
	We use the velocity feedback $u=-2\,\partial_tw$ on $(-1,0)$, the measurement $y=\partial_tw|_{(-1,0)}$ and the collocated injection $L=-2C^*$ which add a damping on the wave part. Numerically, the constants of \cref{A:Kstab,A:Lstab} of the discretized system are stable under mesh refinement. Here $\Acal$ is unbounded and the whole state is monitored. 
	
	The PDE system is discretized by an energy-consistent finite-difference scheme.
	The initial state is given by $\partial_xw_0(x)=\frac\pi2\cos(\frac\pi2(x+1))$, $\partial_tw_0(x)=0$ and $\theta_0(x)=\frac12\sin(\pi x)$, normalized to unit energy, and the observer starts at $\zhat_0=0$.
	
	The event-triggered loops reproduce the decay of the continuous feedback, with inter-event times bounded from below (\cref{fig:num-wh}(a)--(b)).
	
	To isolate the effect of the frequency of the data, we then start the observer at the true state, so that $e_0=0$ and the output residual vanishes, and we take data oscillating at frequency $n$. As for the system of \cref{ex:counter}, $t_1$ decreases like $n^{-1}$ (\cref{fig:num-wh}(c)). Hence there is no dwell-time over bounded sets of data, which reflects the non-uniform continuity of the semigroup (\cref{lem:nonunif}).
	
	\begin{figure}[ht]
		\centering
		\includegraphics[width=\textwidth]{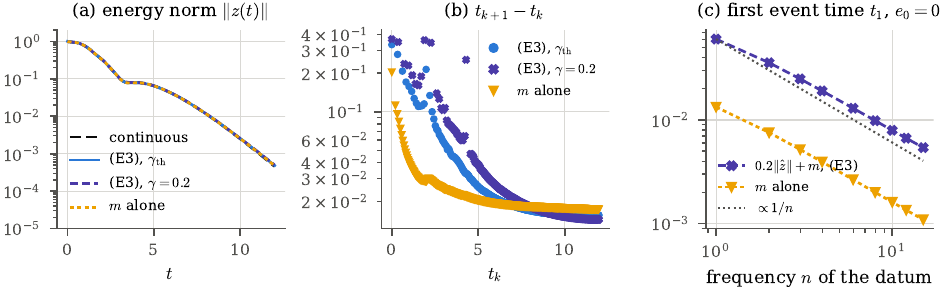}
		\caption{Wave--heat system. (a) Energy of the plant and (b) inter-event times $t_{k+1}-t_k$ against $t_k$, for the initial data mentioned in the text. (c) First event time $t_1$ for the initial data $\zhat_0=z_0$ with $\partial_tw_0$ proportional to $\sin(n\pi(x+1))$, $\partial_xw_0=0$, $\theta_0=0$, normalized to unit energy, so that $e_0=0$.}\label{fig:num-wh}
	\end{figure}
	
	\section{Conclusion
	}\label{sec:conclusion}
	
	We have developed a framework for observer-based event-triggered control of abstract linear systems on Hilbert spaces, under the sole assumptions that the dynamics operator $\Acal$ generates a $C_0$-semigroup, that the control and observation operators are bounded, and that the operators $\Acal+BK$ and $\Acal+LC$ generate exponentially stable semigroups.\\ 
	No structural condition on $\Acal$, such as skew-adjointness, dissipativity or analyticity, is imposed, and the initial data are not restricted to the domain of the generator. Our approach thus accommodates a considerably broader class of physical systems, including transport phenomena, non-conservative damping, and other genuinely non-normal or non-self-adjoint dynamics.\\ 
	Well-posedness and the absence of Zeno behavior hold for any triggering threshold $\Theta$ having the form \eqref{eq:threshold} with  $\gamma\ge0$ and every continuous positive internal variable. Moreover,  exponential stability holds under a small-gain condition on $\gamma$ and an exponential decay of $m$. These results rely on the mild formulation. The positivity of $m$ separates the events, the strong limit of the observer at a finite accumulation time excludes Zeno behavior, and the sampling error is absorbed at the level of the variation of constants formula, so that no differentiability of the trajectory, and no dissipativity of the generator, is invoked.
	
	The dwell-time question has a sharply different answer. All dwell-time estimates rest on the increment identity \eqref{eq:incr-identity}, whose only delicate term is $(e^{h\Acal}-I)\zhat(t_k)$. When the monitored signal has bounded dynamics, as in finite dimension or when a finite-dimensional unstable part is monitored alone, a dwell-time independent of the data and exponential stability hold simultaneously, with an implementable internal variable driven by the measured output residual $LCe$.
	When the generator is unbounded and the whole state is monitored, the events still do not accumulate in finite time and the closed loop remains exponentially stable, and our dwell-time results do not apply (even along a single trajectory). We have exhibited a system satisfying all our assumptions for which no triggering rule can achieve both a dwell-time independent of the data and bounded trajectories. The root cause is that a $C_0$-semigroup with unbounded generator is not continuous at the origin in operator norm.
	
	\section*{Acknowledgments}
	The authors are grateful to Amaury Hayat for insightful discussions at the early stages of this work, as well as for his careful reading of a preliminary version of the manuscript and his many helpful suggestions.
	
	\section*{Disclosure on the use of AI tools}
	During the preparation of this manuscript, the authors used ChatGPT (OpenAI) to assist with the drafting and editing of the text and to help check the mathematical arguments. The authors also used Claude (Anthropic) to assist with the coding. All results, proofs, and conclusions were subsequently reviewed and verified by the authors, who take full responsibility for the content of the paper.

\end{document}